\documentclass{article}
\usepackage{amsmath}
\usepackage{amsfonts}
\usepackage{amssymb}
\usepackage{amsthm}
\usepackage{color}
\usepackage[toc,page]{appendix}

\numberwithin{equation}{section}
\newtheorem{thm}{Theorem}[section]
\newtheorem{prop}[thm]{Proposition}

\begin{document}
\title{Determining the first order perturbation of a bi-harmonic operator on bounded and unbounded domains from partial data}
\author{Yang Yang}
\date{}
\maketitle

\begin{abstract}
In this paper we study inverse boundary value problems with partial data for the bi-harmonic operator with first order perturbation. We consider two types of subsets of $\mathbb{R}^{n}(n\geq 3)$, one is an infinite slab, the other is a bounded domain. In the case of a slab, we show that, from Dirichlet and Neumann data given either on the different boundary hyperplanes of the slab or on the same boundary hyperplane, one can uniquely determine the magnetic potential and the electric potential.

In the case of a bounded domain, we show the unique determination of the magnetic potential and the electric potential
from partial Dirichlet and Neumann data under two different assumptions. The first assumption is that the magnetic and electric potentials are known in a neighborhood of the boundary, in this situation we obtain the uniqueness result when the Dirichlet and Neumann data are only given on two arbitrary open subsets of the boundary. The second assumption is that the Dirichlet and Neumann data are known on the same part of the boundary whose complement is a part of a hyperplane, we also establish the unique determination result in this local data case.
\end{abstract}

\section{Introduction and statement of results}

A bi-harmonic operator with first order perturbation is a differential operator of the form
$$\mathcal{L}_{A,q}(x,D):=\Delta^{2}+A(x)\cdot D+q(x)$$
with $D=\frac{1}{i}\nabla$. Here $A$ is a complex-valued vector field called the magnetic potential, $q$ is a complex-valued function called the electric potential. This type of operators arise in physics when considering the equilibrium configuration of an elastic plate hinged along the boundary. It is also widely used in other physical models, see \cite{GGS}. In this paper we study the identifiability of the first order perturbation of a bi-harmonic operator from partial boundary measurements in two types of open subsets of $\mathbb{R}^{n}$, the first type is an infinite slab, and the second type is a bounded domain with $C^{\infty}$ boundary.

First we consider an infinite slab $\Sigma$. The geometry of an infinite slab arises in many applications, for instance, in the study of wave propagation in marine acoustics. It is also a simple geometric setting in medical imaging. By choosing appropriate coordinates, we may assume that
$$\Sigma:=\{x=(x',x_{n})\in\mathbb{R}^{n}:x'=(x_{1},\dots,x_{n-1})\in\mathbb{R}^{n-1}, 0<x_{n}<L\},\quad L>0.$$
Its boundary consists of two parallel hyperplanes
$$\Gamma_{1}:=\{x\in\mathbb{R}^{n}:x_{n}=L\} \quad\quad \Gamma_{2}:=\{x\in\mathbb{R}^{n}:x_{n}=0\}.$$
Given $(f_{1},f_{2})\in H^{\frac{7}{2}}(\Gamma_{1})\times H^{\frac{3}{2}}(\Gamma_{1})$ with $f_{1}, f_{2}$ compactly supported on $\Gamma_{1}$, we are interested in the following Dirichlet problem
\begin{equation}\label{Dirichlet1}
\left\{
\begin{array}{rll}\vspace{1ex}
\mathcal{L}_{A,q}u= & 0 &\quad\textrm{ in } \Sigma \\ \vspace{1ex}
u=f_{1} \quad \Delta u= & f_{2} & \quad\textrm{ on } \Gamma_{1} \\ \vspace{1ex}
u=0 \quad \Delta u= & 0 & \quad\textrm{ on } \Gamma_{2}. \\
\end{array}
\right.
\end{equation}
In Appendix A we show that problem \eqref{Dirichlet1} has a unique solution in $H^{4}(\Sigma)$, where $H^{4}(\Sigma)$ is the standard Sobolev space on $\Sigma$. We define the Dirichlet-to-Neumann map for the above boundary value problem by
$$
\begin{array}{rcl}\vspace{1ex}
\Lambda_{A,q}: \quad (H^{\frac{7}{2}}(\Gamma_{1})\cap\mathcal{E}'(\Gamma_{1}))\times (H^{\frac{3}{2}}(\Gamma_{1})\cap\mathcal{E}'(\Gamma_{1})) & \rightarrow & H^{\frac{5}{2}}_{loc}(\partial\Sigma)\times H^{\frac{1}{2}}_{loc}(\partial\Sigma)\\ \vspace{1ex}
(f_{1},f_{2}) & \mapsto & (\partial_{\nu}u|_{\partial\Sigma},\partial_{\nu}(\Delta u)|_{\partial\Sigma}),
\end{array}
$$
where $u$ is the solution of \eqref{Dirichlet1}, $\mathcal{E}'(\Gamma_{1})$ is the set of compactly supported distributions on $\Gamma_{1}$, $\nu$ is the unit outer normal vector field to $\partial\Sigma=\Gamma_{1}\cup\Gamma_{2}$. The inverse problem we will study is as follows. Let $\gamma_{1}\subset\Gamma_{1}$,$\gamma_{2}\subset\partial\Sigma$ be non-empty open subsets of the boundary, assuming that
$$\Lambda_{A^{(1)},q^{(1)}}(f_{1},f_{2})|_{\gamma_{2}\times\gamma_{2}}=\Lambda_{A^{(2)},q^{(2)}}(f_{1},f_{2})|_{\gamma_{2}\times\gamma_{2}}$$
for all $(f_{1},f_{2})\in (H^{\frac{7}{2}}(\Gamma_{1})\cap\mathcal{E}'(\Gamma_{1}))\times (H^{\frac{3}{2}}(\Gamma_{1})\cap\mathcal{E}'(\Gamma_{1}))$ with $supp(f_{1})\subset\gamma_{1}$, $supp(f_{2})\subset\gamma_{1}$, can we conclude that $A^{(1)}=A^{(2)}$ and $q^{(1)}=q^{(2)}$ in $\Sigma$?

We will show this is valid for some open subsets $\gamma_{1},\gamma_{2}$ assuming that $A^{(j)},q^{(j)},j=1,2$ are compactly supported in $\bar{\Sigma}$. Our first result considers the case when the data and the measurements are on different boundary hyperplanes.

\begin{thm}
Let $\Sigma\subset\mathbb{R}^{n}(n\geq 3)$ be an infinite slab with boundary hyperplanes $\Gamma_{1}$ and $\Gamma_{2}$. Let $A^{(j)}\in W^{1,\infty}({\Sigma};\mathbb{C}^{n})\cap\mathcal{E}'(\bar{\Sigma};\mathbb{C}^{n})$, $q^{(j)}\in L^{\infty}(\Sigma;\mathbb{C})\cap\mathcal{E}'(\bar{\Sigma};\mathbb{C}), j=1,2.$ Denote by $B\subset\mathbb{R}^{n}$ an open ball containing the supports of $A^{(j)},q^{(j)},j=1,2$. Let $\gamma_{j}\subset\Gamma_{j}$ be open sets such that
$\Gamma_{j}\cap\bar{B}\subset\gamma_{j},\;j=1,2.$ If
$$\Lambda_{A^{(1)},q^{(1)}}(f_{1},f_{2})|_{\gamma_{2}\times\gamma_{2}}=\Lambda_{A^{(2)},q^{(2)}}(f_{1},f_{2})|_{\gamma_{2}\times\gamma_{2}}$$
for all $(f_{1},f_{2})\in (H^{\frac{7}{2}}(\Gamma_{1})\cap\mathcal{E}'(\Gamma_{1}))\times (H^{\frac{3}{2}}(\Gamma_{1})\cap\mathcal{E}'(\Gamma_{1}))$ with $supp(f_{1})\subset\gamma_{1}$ and $supp(f_{2})\subset\gamma_{1}$,
then $A^{(1)}=A^{(2)}$ and $q^{(1)}=q^{(2)}$.
\end{thm}

We would like to remark that when the supports of $A^{(j)},q^{(j)}$ are strictly contained in the interior of the slab, then $\gamma_{1}$ and $\gamma_{2}$ in the above theorem can be chosen to be arbitrarily small.

Our next result considers the case when the data and the measurements are on the same boundary hyperplane.

\begin{thm}
Let $\Sigma\subset\mathbb{R}^{n}(n\geq 3)$ be an infinite slab between two parallel hyperplanes $\Gamma_{1}$ and $\Gamma_{2}$. Let $A^{(j)}\in W^{1,\infty}({\Sigma};\mathbb{C}^{n})\cap\mathcal{E}'(\bar{\Sigma};\mathbb{C}^{n})$, $q^{(j)}\in L^{\infty}(\Sigma;\mathbb{C})\cap\mathcal{E}'(\bar{\Sigma};\mathbb{C}), j=1,2.$ Denote by $B\subset\mathbb{R}^{n}$ an open ball containing the supports of $A^{(j)},q^{(j)},j=1,2$. Let $\gamma_{1},\gamma'_{1}\subset\Gamma_{1}$ be open sets such that $\Gamma_{1}\cap\bar{B}\subset\gamma_{1}$ amd $\Gamma_{1}\cap\bar{B}\subset\gamma'_{1}.$ If
$$\Lambda_{A^{(1)},q^{(1)}}(f_{1},f_{2})|_{\gamma'_{1}\times\gamma'_{1}}=\Lambda_{A^{(2)},q^{(2)}}(f_{1},f_{2})|_{\gamma'_{1}\times\gamma'_{1}}$$
for all $(f_{1},f_{2})\in (H^{\frac{7}{2}}(\Gamma_{1})\cap\mathcal{E}'(\Gamma_{1}))\times (H^{\frac{3}{2}}(\Gamma_{1})\cap\mathcal{E}'(\Gamma_{1}))$ with $supp(f_{1})\subset\gamma_{1}$ and $supp(f_{2})\subset\gamma_{1}$,
then $A^{(1)}=A^{(2)}$ and $q^{(1)}=q^{(2)}$.
\end{thm}

Proofs of Theorem 1.1 and Theorem 1.2 are based on the construction of a special class of complex geometric optics (CGO) solutions which vanish on appropriate boundary hyperplanes, using a reflection argument. The idea of constructing such solutions for the Schr\"{o}dinger operator goes back to \cite{SU}. Constructing complex geometric optics solutions using a reflection argument was initiated in \cite{I}.

Inverse problems of identifying an embedded object in a slab have been studied by many authors in \cite{Ik,KLU,LU,SW}. In \cite{LU} the authors considered the Schr\"{o}dinger operator $\Delta+q$ in a slab and showed that the electric potential $q$ can be uniquely determined from partial boundary measurements. In \cite{KLU} the authors considered the magnetic Schr\"{o}dinger operator $\Delta+A(x)\cdot D+q$ and showed that the magnetic field $dA$ and the electric potential $q$ can be uniquely determined from partial boundary measurements. Here $dA$ is the exterior differentiation of the magnetic potential vector field $A$, and notice that determining $dA$ is equivalent to determining the equivalence class $\{\tilde{A}:\tilde{A}=A+\nabla\Phi \textrm{ for some }\Phi\in C^{1,1}(\overline{\Sigma})\}$. It was also pointed out in \cite{KLU} that, by only looking at the Dirichlet-to-Neumann map, such a gauge transformation obstruction always exists, so the best one can hope for the magnetic Schr\"{o}dinger operator is to determine $dA$. However, for the perturbed bi-harmonic operator, our results indicate that this type of obstruction can be overcome and one therefore determines not only $dA$, but also $A$ itself. This is due to the fact that in our proof we are able to construct more CGO solutions than for the magnetic Schr\"{o}dinger operator thanks to the higher order of the bi-harmonic operator.

\vspace{1ex}
In the remaining part of this section we shall discuss an inverse boundary value problem for the perturbed bi-harmonic operator on a bounded domain. Let $\Omega\subset\mathbb{R}^{n}(n\geq 3)$ be a bounded open subset with $C^{\infty}$ boundary. Consider the Dirichlet problem
\begin{equation}\label{Dirichlet2}
\left\{
\begin{array}{rll}\vspace{1ex}
\mathcal{L}_{A,q}u= & 0 &\quad\textrm{ in } \Omega \\ \vspace{1ex}
u= & f_{1} & \quad\textrm{ on } \partial\Omega \\ \vspace{1ex}
\Delta u= & f_{2} & \quad\textrm{ on } \partial\Omega \\
\end{array}
\right.
\end{equation}
with $A\in W^{1,\infty}(\Omega;\mathbb{C}^{n}),q\in L^{\infty}(\Omega;\mathbb{C})$ and $(f_{1},f_{2})\in H^{\frac{7}{2}}(\partial\Omega)\times H^{\frac{3}{2}}(\partial\Omega)$. The operator $\mathcal{L}_{A,q}$, equipped with the domain
$$\mathcal{D}(\mathcal{L}_{A,q}):=\{u\in H^{4}(\Omega):u|_{\partial\Omega}=(\Delta u)|_{\partial\Omega}=0\}$$
is an unbounded closed operator on $L^{2}(\Omega)$ with purely discrete spectrum, see \cite{G}. We make the following assumption
\begin{center}
\textbf{(A)}:  $0$ is not an eigenvalue of the perturbed bi-harmonic operator $\mathcal{L}_{A,q}:\mathcal{D}(\mathcal{L}_{A,q})\rightarrow L^{2}(\Omega)$.
\end{center}
Under the assumption \textbf{(A)}, the Dirichlet problem \eqref{Dirichlet2} has a unique solution $u\in H^{4}(\Omega)$, Let $\nu$ be the unit outer normal vector field to $\partial\Omega$, we define the Dirichlet-to-Neumann map to \eqref{Dirichlet2} as
$$
\begin{array}{rl}\vspace{1ex}
\Lambda_{A,q}:H^{\frac{7}{2}}(\partial\Omega)\times H^{\frac{3}{2}}(\partial\Omega) & \rightarrow H^{\frac{5}{2}}(\partial\Omega)\times H^{\frac{1}{2}}(\partial\Omega)\\ \vspace{1ex}
(f_{1},f_{2}) & \mapsto (\partial_{\nu}u|_{\partial\Omega},\partial_{\nu}(\Delta u)|_{\partial\Omega})
\end{array}
$$
where $u\in H^{4}(\Omega)$ is the solution to the problem \eqref{Dirichlet2}. We can also introduce the Cauchy data set $\mathcal{C}_{A,q}$ for the operator $\mathcal{L}_{A,q}$ defined by
$$\mathcal{C}_{A,q}:=\{(u|_{\partial\Omega},(\Delta u)|_{\partial\Omega},\partial_{\nu}u|_{\partial\Omega},\partial_{\nu}(\Delta u)|_{\partial\Omega}):u\in H^{4}(\Omega), \mathcal{L}_{A,q}u=0 \textrm{ in } \Omega\}.$$
When the assumption \textbf{(A)} holds, the Cauchy data set $\mathcal{C}_{A,q}$ is the graph of the Dirichlet-to-Neumann map $\Lambda_{A,q}$.

Let $\gamma_{1},\gamma_{2}\subset\partial\Omega$ be non-empty open subsets of the boundary. In this paper we are interested in the inverse boundary value problem for the operator $\mathcal{L}_{A,q}$ with partial boundary measurements: assuming that
$$\Lambda_{A^{(1)},q^{(1)}}(f_{1},f_{2})|_{\gamma_{2}\times\gamma_{2}}=\Lambda_{A^{(2)},q^{(2)}}(f_{1},f_{2})|_{\gamma_{2}\times\gamma_{2}}$$
for all $(f_{1},f_{2})\in H^{\frac{7}{2}}(\partial\Omega)\times H^{\frac{3}{2}}(\partial\Omega)$ with $supp(f_{1})\subset\gamma_{1}$ and $supp(f_{2})\subset\gamma_{1}$, can we conclude that $A^{(1)}=A^{(2)}$ and $q^{(1)}=q^{(2)}$ in $\Omega$?

For the bi-harmonic operator, determination of the first order perturbation on a bounded domain $\Omega$ was considered in \cite{KLU3} with partial boundary measurements. The authors showed that, from the Dirichlet-to-Neumann map, one can uniquely determine not only the electric potential $q$, but also the magnetic potential $A$. Again this is different from the situation for the magnetic Schr\"{o}dinger operator where the gauge transformation exists as an obstruction for the recovery of $A$. In this paper, we will improve the uniqueness result in \cite{KLU3} under two different assumptions: in Theorem 1.3, we assume $A^{(1)}=A^{(2)}$ and $q^{(1)}=q^{(2)}$ in a neighborhood of $\partial\Omega$, and show that we still have the uniqueness even when both $\gamma_{1}$ and $\gamma_{2}$ are arbitrarily small; in Theorem 1.4, we assume the inaccessible part of the boundary is contained in a plane, and prove the uniqueness with local data.

\begin{thm}
Let $\Omega\subset\mathbb{R}^{n}(n\geq 3)$ be a bounded domain with $C^{\infty}$ connected boundary. Let $A^{(j)}\in W^{1,\infty}({\Omega};\mathbb{C}^{n})$, $q^{(j)}\in L^{\infty}(\Omega;\mathbb{C}), j=1,2$ be such that the assumption \textbf{(A)} holds for both operators. Assume that $A^{(1)}=A^{(2)}$ and $q^{(1)}=q^{(2)}$ in a neighborhood of the boundary $\partial\Omega$. Let $\gamma_{1},\gamma_{2}\subset\partial\Omega$ be non-empty open subsets of the boundary.
If
$$\Lambda_{A^{(1)},q^{(1)}}(f_{1},f_{2})|_{\gamma_{2}\times\gamma_{2}}=\Lambda_{A^{(2)},q^{(2)}}(f_{1},f_{2})|_{\gamma_{2}\times\gamma_{2}}$$
for all $(f_{1},f_{2})\in (H^{\frac{7}{2}}(\Gamma_{1})\cap\mathcal{E}'(\Gamma_{1}))\times (H^{\frac{3}{2}}(\Gamma_{1})\cap\mathcal{E}'(\Gamma_{1}))$ with $supp(f_{1})\subset\gamma_{1}$ and $supp(f_{2})\subset\gamma_{1}$, then $A^{(1)}=A^{(2)}$ and $q^{(1)}=q^{(2)}$ in $\Omega$.
\end{thm}

In the following theorem, notice that we need the magnetic potential $A$ and electric potential $q$ to be smooth. This is due to the fact that our proof relies on determination of the boundary value of $A$ from the Cauchy data set $\mathcal{C}_{A,q}$. For the bi-harmonic operator, or more generally for poly-harmonic operators, this result was proved only for smooth $A$ and $q$ in \cite{KLU2}.

\begin{thm}
Let $\Omega\subset\{x\in\mathbb{R}^{n}:x_{n}>0\}(n\geq 3)$ be a bounded domain with $C^{\infty}$ connected boundary, and let $\partial\Omega\cap\{x\in\mathbb{R}^{n}:x_{n}=0\}\neq\emptyset$ and $\gamma:=\partial\Omega\backslash\{x\in\mathbb{R}^{n}:x_{n}=0\}$. Let $A^{(j)}\in C^{\infty}(\overline{\Omega};\mathbb{C}^{n})$, $q^{(j)}\in C^{\infty}(\overline{\Omega};\mathbb{C}), j=1,2$ be such that the assumption \textbf{(A)} holds for both operators.
If
$$\Lambda_{A^{(1)},q^{(1)}}(f_{1},f_{2})|_{\gamma\times\gamma}=\Lambda_{A^{(2)},q^{(2)}}(f_{1},f_{2})|_{\gamma\times\gamma}$$
for all $(f_{1},f_{2})\in H^{\frac{7}{2}}(\partial\Omega)\times H^{\frac{3}{2}}(\partial\Omega)$ with $supp(f_{1})\subset\bar{\gamma}$ and $supp(f_{2})\subset\bar{\gamma}$,
then $A^{(1)}=A^{(2)}$ and $q^{(1)}=q^{(2)}$ in $\Omega$.
\end{thm}

This paper is structured as follows: in Section 2 we establish a Carleman type estimate for the bi-harmonic operator and then construct a class of CGO solutions on a bounded domain; in Section 3 we show an integral identity and a Runge type approximation theorem; in Section 4 we construct the CGO solutions we desire in the slab by reflecting the CGO solutions constructed in Section 2; Section 5, 6, 7, 8 are devoted to the proof of Theorem 1.1, 1.2, 1,3 and 1.4 respectively. In Appendix A we prove the solvability of the boundary value problem \eqref{Dirichlet1} and some identities used in the proofs of the main theorems.

\section{Carleman estimate and CGO solutions on a bounded domain}
In this section we construct some CGO solutions on a bounded domain to the equation $\mathcal{L}_{A,q}u=0$. CGO solutions have been intensively utilized in establishing uniqueness result in elliptic inverse boundary value problems. For the construction of various CGO solutions and their application, see \cite{BRUZ,CY,CY2,FKSU,KSU,SU}.

Let $\Omega\subset\mathbb{R}^{n}$, $n\geq 3$, be a bounded domain with $C^{\infty}$ boundary. Consider the equation $\mathcal{L}_{A,q}u=0$ in $\Omega$ with $A\in W^{1,\infty}(\Omega;\mathbb{C}^{n})$ and $q\in L^{\infty}(\Omega;\mathbb{C})$. We will construct CGO solutions of the form
\begin{equation}\label{form}
u(x,\zeta,h)=e^{x\cdot\zeta/h}(a(x,\zeta)+r(x,\zeta,h)).
\end{equation}
based on a Carleman estimate. Here $\zeta\in\mathbb{C}^{n}$ is a complex vector satisfying $\zeta\cdot\zeta=0$, $a$ is a smooth amplitude, $r$ is a correction term, $h>0$ is a small semiclassical parameter. To deal with the perturbation, we extend $A\in W^{1,\infty}(\Omega;\mathbb{C}^{n})$ to a Lipschitz vector field compactly supported in $\mathbb{R}^{n}$, extend $q\in L^{\infty}(\Omega;\mathbb{C})$ as zero to $\mathbb{R}^{n}$. We shall work with $\zeta$ depending slightly on $h$, i.e. $\zeta=\zeta^{(0)}+\zeta^{(1)}$ with $\zeta^{(0)}$ independent of $h$, $\zeta^{(1)}=\mathcal{O}(h)$, and $|\textrm{Re } \zeta^{(0)}|=|\textrm{Im } \zeta^{(0)}|=1$.

Consider the conjugated operator
$$
h^{4}e^{-x\cdot\zeta/h}\mathcal{L}_{A,q}e^{x\cdot\zeta/h}=(h^{2}\Delta+2ih\zeta\cdot\nabla)^{2}+h^{3}A\cdot hD-ih^{3}A\cdot \zeta+h^{4}q
$$
In order to eliminate the lowest order term involving $h$ in this expression, we require
\begin{equation}\label{eq_a}
(\zeta^{(0)}\cdot\nabla)^{2}a=0 \quad\quad \textrm{ in } \Omega.
\end{equation}
As $|\textrm{Re } \zeta^{(0)}|=|\textrm{Im } \zeta^{(0)}|=1$, $\zeta^{(0)}\cdot\nabla$ is a $\bar{\partial}$-operator in appropriate coordinates, so the above equation admits a solution $a=a(x,\zeta^{(0)})\in C^{\infty}(\overline{\Omega})$. To find an appropriate correction term, we need a Carleman type estimate. We will use the semiclassical Sobolev spaces $H^{s}_{\textrm{scl}}(\mathbb{R}^{n})$ ($s\in\mathbb{R}$) with the norm $\|f\|_{H^{s}_{\textrm{scl}}(\mathbb{R}^{n})}=\|\langle hD\rangle^{s}f\|_{L^{2}(\mathbb{R}^{n})}$ where $\langle\xi\rangle=(1+|\xi|^{2})^{\frac{1}{2}}$.

\begin{prop}\label{Carleman}
Suppose $A\in W^{1,\infty}(\Omega;\mathbb{C}^{n})$, $q\in L^{\infty}(\Omega;\mathbb{C})$. Then for $h>0$ sufficiently small, there exists a constant $C>0$ independent of $h$ such that
$$\|u\|_{L^{2}(\mathbb{R}^{n})}\leq Ch^{2}\|e^{x\cdot\zeta/h}\mathcal{L}_{A,q}e^{-x\cdot\zeta/h}u\|_{H^{-1}_{\textrm{scl}}(\mathbb{R}^{n})} \quad\quad u\in C^{\infty}_{c}(\Omega).$$
\end{prop}
\begin{proof}
From \cite[Proposition 4.2]{KS}, we can find a constant $C_{1}>0$ independent of $h$ such that for all $u\in C^{\infty}_{c}(\Omega)$
$$
\|u\|_{L^{2}(\mathbb{R}^{n})}\leq C_{1}h\|e^{x\cdot\zeta/h}\Delta e^{-x\cdot\zeta/h}u\|_{H^{-1}_{\textrm{scl}}(\mathbb{R}^{n})}
$$
Iterate to get
\begin{equation}\label{CE}
\|u\|_{L^{2}(\mathbb{R}^{n})}\leq C^{2}_{1}h^{2}\|e^{x\cdot\zeta/h}\Delta^{2} e^{-x\cdot\zeta/h}u\|_{H^{-1}_{\textrm{scl}}(\mathbb{R}^{n})}
\end{equation}
We can add the zeroth order term $h^{2}q$ to \eqref{CE} since
$$
h^{2}\|qu\|_{H^{-1}_{\textrm{scl}}(\mathbb{R}^{n})}\leq h^{2}\|qu\|_{L^{2}(\mathbb{R}^{n})}\leq h^{2}\|q\|_{L^{\infty}(\mathbb{R}^{n})}\|u\|_{L^{2}(\mathbb{R}^{n})}.
$$
We can add the first order term $h^{2}e^{x\cdot\zeta/h}A\cdot De^{-x\cdot\zeta/h}$ as
$$
h^{2}e^{x\cdot\zeta/h}A\cdot De^{-x\cdot\zeta/h}=h(A\cdot hD+iA\cdot\zeta)
$$
and
$$h\|A\cdot\zeta u\|_{H^{-1}_{scl}(\mathbb{R}^{n})}\leq h\|A\cdot\zeta u\|_{L^{2}(\mathbb{R}^{2})}\leq h\|A\cdot\zeta\|_{L^{\infty}(\mathbb{R}^{n})}\|u\|_{L^{2}(\mathbb{R}^{n})}.$$
$$
\begin{array}{rl}\vspace{1ex}
h\|A\cdot hDu\|_{H^{-1}_{\textrm{scl}}(\mathbb{R}^{n})}\leq & h\displaystyle\sum^{n}_{j=1}\|hD_{j}(A_{j}u)\|_{H^{-1}_{\textrm{scl}}(\mathbb{R}^{n})}+\mathcal{O}(h^{2})\|(\textrm{div} A)u\|_{H^{-1}_{\textrm{scl}}(\mathbb{R}^{n})}\\ \vspace{1ex}
\leq & \mathcal{O}(h)\displaystyle\sum^{n}_{j=1}\|A_{j}u\|_{L^{2}(\mathbb{R}^{n})}+\mathcal{O}(h^{2})\|u\|_{L^{2}(\mathbb{R}^{2})} \\ \vspace{1ex}
\leq & \mathcal{O}(h)\|u\|_{L^{2}(\mathbb{R}^{n})}.
\end{array}
$$
After adding these perturbation terms, we get the desired result.
\end{proof}

Denote $\|f\|^{2}_{H^{1}_{scl}(\Omega)}:=\|f\|^{2}_{L^{2}(\Omega)}+h^{2}\|\nabla f\|^{2}_{L^{2}(\Omega)}$. The following solvability result is an immediate consequence of the above Carleman estimate and the Hahn-Banach Theorem.
\begin{prop}\label{solvability}
Suppose $A\in W^{1,\infty}(\Omega;\mathbb{C}^{n})$, $q\in L^{\infty}(\Omega;\mathbb{C})$. Then for any $f\in L^{2}(\Omega)$, the equation
$$e^{-x\cdot\zeta/h}\mathcal{L}_{A,q}e^{x\cdot\zeta/h}r=f \quad\quad \textrm{ in }\Omega$$
has a solution $r\in H^{1}(\Omega)$ with $\|r\|_{H^{1}_{\textrm{scl}}(\Omega)}\leq\mathcal{O}(h^{2})\|f\|_{L^{2}(\Omega)}$.
\end{prop}
\begin{proof}
We extend $A$ to a compactly supported Lipschitz vector field in $\mathbb{R}^{n}$, extend $q$ and $f$ as zero, and solve the equation in $\mathbb{R}^{n}$. Denote $\mathcal{L}_{\zeta}:=e^{-x\cdot\zeta/h}\mathcal{L}_{A,q}e^{x\cdot\zeta/h}$, the $L^{2}$-adjoint of $\mathcal{L}_{\zeta}$ is given by
$$\mathcal{L}^{\ast}_{\zeta}:=e^{x\cdot\zeta/h}\mathcal{L}^{\ast}_{A,q}e^{-x\cdot\zeta/h}=e^{x\cdot\zeta/h}
\mathcal{L}_{\bar{A},i^{-1}\nabla\cdot\bar{A}+\bar{q}}e^{-x\cdot\zeta/h}.$$
Consider the complex linear functional
$$L:\mathcal{L}^{\ast}_{\zeta}C^{\infty}_{c}(\Omega)\rightarrow\mathbb{C} \quad\quad \mathcal{L}^{\ast}_{\zeta}u\mapsto (u,f)_{L^{2}(\mathbb{R}^{n})}.$$
Applying Proposition \ref{Carleman} with $\mathcal{L}_{A,q}$ replaced by $\mathcal{L}^{\ast}_{A,q}$, we see the map $L$ is well-defined and for any $u\in C^{\infty}_{c}(\Omega)$, we have
$$
|L(\mathcal{L}^{\ast}_{\zeta}u)|=|(u,f)_{L^{2}(\mathbb{R}^{n})}|\leq \|u\|_{L^{2}(\mathbb{R}^{n})}\|f\|_{L^{2}(\mathbb{R}^{n})}\leq Ch^{2}\|\mathcal{L}^{\ast}_{\zeta}u\|_{H^{-1}_{\textrm{scl}}(\mathbb{R}^{n})}\|f\|_{L^{2}(\mathbb{R}^{n})}.
$$
This shows that $L$ is bounded in the $H^{-1}(\mathbb{R}^{n})$-norm. By the Hahn-Banach theorem we can extend $L$ to a bounded linear functional $\tilde{L}$ on $H^{-1}(\mathbb{R}^{n})$ without increasing the norm. Thus, by Riesz representation theorem, there exists $r\in H^{1}(\mathbb{R}^{n})$ such that for all $u\in H^{-1}(\mathbb{R}^{n})$ we have
$$\tilde{L}(u)=(u,r)_{H^{-1}(\mathbb{R}^{n}),H^{1}(\mathbb{R}^{n})} \quad \textrm{ and } \quad \|r\|_{H^{1}_{\textrm{scl}}(\mathbb{R}^{n})}\leq Ch^{2}\|f\|_{L^{2}(\mathbb{R}^{n})}.$$
Here $(u,r)_{H^{-1}(\mathbb{R}^{n}),H^{1}(\mathbb{R}^{n})}$ stands for the $L^{2}$-duality. It follows that $\mathcal{L}_{\zeta}r=f$ in $\mathbb{R}^{n}$, hence also in $\Omega$, and $\|r\|_{H^{1}_{scl}(\Omega)}\leq
\|r\|_{H^{1}_{scl}(\mathbb{R}^{n})}\leq Ch^{2}\|f\|_{L^{2}(\mathbb{R}^{n})}=Ch^{2}\|f\|_{L^{2}(\Omega)}$.
\end{proof}

Now we can complete the construction of the CGO solution in \eqref{form}. Equation \eqref{eq_a} gives $e^{-x\cdot\zeta/h}\mathcal{L}_{A,q}e^{x\cdot\zeta/h}a=\mathcal{O}(h^{-1})$. From Proposition \ref{solvability}, we can find $r\in H^{1}(\Omega)$ with $\|r\|_{H^{1}_{\textrm{scl}}(\Omega)}=\mathcal{O}(h)$ such that
$$
e^{-x\cdot\zeta/h}\mathcal{L}_{A,q}e^{x\cdot\zeta/h}r=-e^{-x\cdot\zeta/h}\mathcal{L}_{A,q}e^{x\cdot\zeta/h}a.
$$
Summing up, we have proved
\begin{prop}\label{existence}
Let $A\in W^{1,\infty}(\Omega;\mathbb{C}^{n})$, $q\in L^{\infty}(\Omega;\mathbb{C})$, and $\zeta\in\mathbb{C}^{n}$ be such that $\zeta\cdot\zeta=0$. Then for all $h>0$ small enough, there exist solutions $u\in H^{1}(\Omega)$ to the equation $\mathcal{L}_{A,q}u=0$ in $\Omega$ of the form
$$u(x,\zeta,h)=e^{x\cdot\zeta/h}(a(x,\zeta^{(0)})+r(x,\zeta,h)),$$
where $a\in C^{\infty}(\overline{\Omega})$ satisfies \eqref{eq_a} and $\|r\|_{H^{1}_{\textrm{scl}}(\Omega)}=\mathcal{O}(h)$.
\end{prop}

\noindent\textbf{Remark:} Sometimes we may need complex geometric optics solutions belonging to $H^{4}(\Omega)$, we can obtain such solutions as follows. Let $\Omega'\supset\supset\Omega$ be a bounded domain with smooth boundary. Extend $A\in W^{1,\infty}(\Omega;\mathbb{C}^{n})$ and $q\in L^{\infty}(\Omega;\mathbb{C})$ to functions in $W^{1,\infty}(\Omega';\mathbb{C}^{n})$ and $L^{\infty}(\Omega';\mathbb{C})$, respectively. By elliptic regularity, the complex geometric optics solutions constructed as above in $\Omega'$ will belong to $H^{4}(\Omega)$.

\section{Integral identity and Runge approximation}
For the bi-harmonic operator, Green's formula gives
\begin{equation} \label{green}
\begin{array}{ll} \vspace{1ex}
&\displaystyle\int_{\Omega}(\mathcal{L}_{A,q}u)\bar{v}\,dx-\int_{\Omega}u\overline{\mathcal{L}^{\ast}_{A,q}v}\,dx=-i\int_{\partial\Omega}\nu(x)\cdot Au\bar{v}\,dS-\int_{\partial\Omega}\partial_{\nu}(-\Delta u)\bar{v}\,dS \\
&+\displaystyle\int_{\partial\Omega}(-\Delta u)\overline{\partial_{\nu}v}\,dS-\int_{\partial\Omega}\partial_{\nu}u\overline{(-\Delta v)}\,dS+\int_{\partial\Omega}u\overline{(\partial_{\nu}(-\Delta v))}\,dS\\
\end{array}
\end{equation}
for all $u,v\in H^{4}(\Omega)$. Here $\mathcal{L}^{\ast}_{A,q}:=\mathcal{L}_{\bar{A},i^{-1}\nabla\cdot\bar{A}+\bar{q}}$ is the adjoint of $\mathcal{L}_{A,q}$, $\nu$ is the unit outer normal vector to the boundary $\partial\Omega$, and $dS$ is the surface measure on $\partial\Omega$.

For $(f_{1},f_{2})\in (H^{\frac{7}{2}}(\Gamma_{1})\cap\mathcal{E}'(\Gamma_{1}))\times (H^{\frac{3}{2}}(\Gamma_{1})\cap\mathcal{E}'(\Gamma_{1}))$ with supp$(f_{1})\subset\gamma_{1}$ and supp$(f_{2})\subset\gamma_{1}$, let $u_{1}\in H^{4}(\Sigma)$ solve
$$\left\{
\begin{array}{rcll}
\mathcal{L}_{A^{(1)},q^{(1)}}u_{1} &=& 0 & \quad\textrm{ in } \Sigma \\
u_{1}=f_{1} \quad\quad \Delta u_{1}&=& f_{2} & \quad\textrm{ on } \Gamma_{1} \\
u_{1}=0 \quad\quad \Delta u_{1}&=& 0 & \quad\textrm{ on } \Gamma_{2} \\
\end{array}
\right.$$
Let $u_{2}\in H^{4}(\Sigma)$ solve
$$\left\{
\begin{array}{rcll}
\mathcal{L}_{A^{(2)},q^{(2)}}u_{2} &=& 0 & \quad\textrm{ in } \Sigma \\
u_{2}=u_{1} \quad\quad \Delta u_{2}&=& \Delta u_{1} & \quad\textrm{ on } \Gamma_{1}\cup\Gamma_{2} \\
\end{array}
\right.$$
Let $w:=u_{2}-u_{1}$, then
\begin{equation} \label{difference}
\mathcal{L}_{A^{(2)},q^{(2)}}w=(A^{(1)}-A^{(2)})\cdot Du_{1}+(q^{(1)}-q^{(2)})u_{1}.
\end{equation}
Suppose $\Lambda_{A^{(1)},q^{(1)}}(f_{1},f_{2})|_{\gamma_{2}}=\Lambda_{A^{(2)},q^{(2)}}(f_{1},f_{2})|_{\gamma_{2}}$, then
\begin{equation}\label{normal}
\partial_{\nu}u_{1}|_{\gamma_{2}}=\partial_{\nu}u_{2}|_{\gamma_{2}}, \quad\quad \partial_{\nu}(\Delta u_{1})|_{\gamma_{2}}=\partial_{\nu}(\Delta u_{2})|_{\gamma_{2}},
\end{equation}
from which we conclude $\partial_{\nu}w=\partial_{\nu}(\Delta w)=0$ on $\gamma_{2}$. We denote
$$l_{1}:=\Gamma_{1}\cap\overline{B}\subset\gamma_{1}, \quad l_{2}:=\Gamma_{2}\cap\overline{B}\subset\gamma_{2}, \quad l_{3}:=\Sigma\cap\partial B.$$
Apparently $\partial(\Sigma\cap B)=l_{1}\cup l_{2}\cup l_{3}$. It follows from \eqref{difference} that $w\in H^{4}(\Sigma)$ is a solution to
$$\Delta^{2}w=0 \quad \textrm{ in } \quad \Sigma\backslash\overline{B}.$$
As $w=\partial_{\nu}w=0$ on $\gamma_{2}\backslash\overline{l}_{2}$, by unique continuation, $w=0$ in $\Sigma\backslash\overline{B}$. Therefore $w=\Delta w=\partial_{\nu}w=\partial_{\nu}(\Delta w)=0$ on $l_{3}$. We record these results here:
\begin{equation} \label{results}
\begin{array}{cl}
w=0 & \textrm{ on } l_{1}\\
w=\partial_{\nu}w=\partial_{\nu}(\Delta w)=0 & \textrm{ on } l_{2} \\
w=\partial_{\nu}w=\Delta w=\partial_{\nu}(\Delta w)=0 & \textrm{ on } l_{3} \\
\end{array}
\end{equation}

If $v$ is a solution of the equation
\begin{equation}\label{eqn_v}
\mathcal{L}^{\ast}_{A^{(2)},q^{(2)}}v=0 \quad\textrm{ in } \Sigma\cap B
\end{equation}
such that
\begin{equation}\label{condition_v}
v=\Delta v=0 \quad \textrm{ on } l_{1}.
\end{equation}
Taking into consideration of $\eqref{difference}$ and $\eqref{eqn_v}$, we apply \eqref{green} to $w$ and $v$ over $\Sigma\cap B$ to get
\begin{equation} \label{identity1}
\begin{array}{ll} \vspace{1ex}
&\displaystyle\int_{\Sigma\cap B}((A^{(1)}-A^{(2)})\cdot Du_{1})\bar{v}\,dx+\int_{\Sigma\cap B}(q^{(1)}-q^{(2)})u_{1}\bar{v}\,dx \\ \vspace{1ex}
=& -i\displaystyle\int_{\partial(\Sigma\cap B)}\nu(x)\cdot A^{(2)}w\bar{v}\,dS-\int_{\partial(\Sigma\cap B)}\partial_{\nu}(-\Delta w)\bar{v}\,dS \\ \vspace{1ex}
&+\displaystyle\int_{\partial(\Sigma\cap B)}(-\Delta w)\overline{\partial_{\nu}v}\,dS-\int_{\partial(\Sigma\cap B)}\partial_{\nu}w\overline{(-\Delta v)}\,dS \\ \vspace{1ex}
&+\displaystyle\int_{\partial(\Sigma\cap B)}w\overline{(\partial_{\nu}(-\Delta v))}\,dS\\
:=&I_{1}+I_{2}+I_{3}+I_{4}+I_{5}
\end{array}
\end{equation}
We analyze each term on the right-hand side and show $I_{j}=0, j=1,\cdots,5$.

$$I_{1}:=-i\displaystyle\int_{\partial(\Sigma\cap B)}\nu(x)\cdot A^{(2)}w\bar{v}\,dS.$$
By \eqref{results}, $w=0$ on $\partial(\Sigma\cap B)$; hence $I_{1}=0$.

$$I_{2}:=-\displaystyle\int_{\partial(\Sigma\cap B)}\partial_{\nu}(-\Delta w)\bar{v}\,dS.$$
By \eqref{condition_v}, $v=0$ on $l_{1}$; by \eqref{results}, $\partial_{\nu}(\Delta w)=0$ on $l_{2}\cup l_{3}$; hence $I_{2}=0$.

$$I_{3}:=\displaystyle\int_{\partial(\Sigma\cap B)}(-\Delta w)\overline{\partial_{\nu}v}\,dS.$$
By definition, $\Delta w=0$ on $l_{1}\cup l_{2}$; by \eqref{results}, $\Delta w=0$ on $l_{3}$; hence $I_{3}=0$.

$$I_{4}:=-\displaystyle\int_{\partial(\Sigma\cap B)}\partial_{\nu}w\overline{(-\Delta v)}\,dS.$$
By \eqref{condition_v}, $\Delta v=0$ on $l_{1}$; by \eqref{results}, $\partial_{\nu}w=0$ on $l_{2}\cup l_{3}$; hence $I_{4}=0$.

$$I_{5}:=\displaystyle\int_{\partial(\Sigma\cap B)}w\overline{(\partial_{\nu}(-\Delta v))}\,dS.$$
By definition, $w=0$ on $l_{1}\cup l_{2}$; by \eqref{results}, $w=0$ on $l_{3}$; hence $I_{5}=0$.

Putting these together, from \eqref{identity1} we obtain
\begin{equation} \label{identity2}
\displaystyle\int_{\Sigma\cap B}((A^{(1)}-A^{(2)})\cdot Du_{1})\bar{v}\,dx+\int_{\Sigma\cap B}(q^{(1)}-q^{(2)})u_{1}\bar{v}\,dx=0.
\end{equation}
for all $u_{1}\in \mathcal{W}(\Sigma)$ and $v\in \mathcal{V}_{l_{1}}(\Sigma\cap B)$. Here for $j=1,2,$ we define some function spaces for later use:
$$
\begin{array}{rl} \vspace{1ex}
\mathcal{W}(\Sigma):=&\{u\in H^{4}(\Sigma):\mathcal{L}_{A^{(1)},q^{(1)}}u=0 \textrm{ in } \Sigma, u|_{\Gamma_{2}}=\Delta u|_{\Gamma_{2}}=0,\\
 & \textrm{ supp}(u|_{\Gamma_{1}})\subset\gamma_{1}, \textrm{ supp}(\Delta u|_{\Gamma_{1}})\subset\gamma_{1}\}.\\
\end{array}
$$
$$
\mathcal{V}_{l_{j}}(\Sigma\cap B):=\{v\in H^{4}(\Sigma\cap B):\mathcal{L}^{\ast}_{A^{(2)},q^{(2)}}v=0 \textrm{ in } \Sigma\cap B, v|_{l_{j}}=\Delta v|_{l_{j}}=0\}.
$$
$$
\mathcal{W}_{l_{j}}(\Sigma\cap B):=\{u\in H^{4}(\Sigma\cap B):\mathcal{L}_{A^{(1)},q^{(1)}}u=0 \textrm{ in } \Sigma\cap B, u|_{l_{j}}=\Delta u|_{l_{j}}=0\}.
$$
We would like to replace $u_{1}\in \mathcal{W}(\Sigma)$ in \eqref{identity2} by elements of the space $\mathcal{W}_{l_{2}}(\Sigma\cap B)$. This can be achieved by the following Runge type approximation result.

\begin{prop}\label{runge}
$\mathcal{W}(\Sigma)$ is a dense subspace of $\mathcal{W}_{l_{2}}(\Sigma\cap B)$ in $L^{2}(\Sigma\cap B)$ topology.
\end{prop}
\begin{proof}
It suffices to establish the following fact: for any $g\in L^{2}(\Sigma\cap B)$ such that
$$\displaystyle\int_{\Sigma\cap B} u\overline{g}\,dx=0 \quad\quad \forall u\in \mathcal{W}(\Sigma),$$
we have
$$\displaystyle\int_{\Sigma\cap B} v\overline{g}\,dx=0 \quad\quad \forall v\in \mathcal{W}_{l_{2}}(\Sigma\cap B).$$
To prove this fact, we extend $g$ by zero to $\Sigma\backslash\Sigma\cap B$. Let $U\in H^{4}(\Sigma)$ be the solution of the problem
$$
\begin{array}{rl} \vspace{1ex}
\mathcal{L}^{\ast}_{A^{(1)},q^{(1)}}U=g & \quad \textrm{ in } \Sigma\\ \vspace{1ex}
U=\Delta U=0 & \quad \textrm{ on } \Gamma_{1}\cup\Gamma_{2}.\\
\end{array}
$$
For any $u\in \mathcal{W}(\Sigma)$, Green's formula in the infinite slab $\Sigma$ (see appendix B) gives
$$
0=\displaystyle\int_{\Sigma}u\overline{g}\,dx=\int_{\Sigma}u\overline{\left(\mathcal{L}^{\ast}_{A^{(1)},q^{(1)}}U\right)}\,dx
=\int_{\Gamma_{1}}\overline{\partial_{\nu}U}\Delta u\,dS+\int_{\Gamma_{1}}\overline{\partial_{\nu}\Delta U} u\,dS.
$$
Since $u|_{\Gamma_{1}}$ and $\Delta u|_{\Gamma_{1}}$ can be arbitrary smooth functions supported in $\gamma_{1}$, we conclude that $\partial_{\nu}U|_{\gamma_{1}}=\partial_{\nu}\Delta U|_{\gamma_{1}}=0$. Hence $U$ satisfies $\Delta^{2} U=0$ in $\Sigma\backslash B$, and moreover, $U=\partial_{\nu}U=0$ on $\gamma_{1}\backslash l_{1}$. Thus, by unique continuation, $U=0$ in $\Sigma\backslash B$, and we have $U=\partial_{\nu} U=0$ on $l_{3}$. Similarly $\Delta U$ satisfies $\Delta(\Delta U)=0$ in $\Sigma\backslash B$ and $\Delta U=\partial_{\nu}\Delta U=0$ on $\gamma_{1}\backslash l_{1}$. Again by unique continuation, $\Delta U=0$ in $\Sigma\backslash B$, and we have $\Delta U=\partial_{\nu}\Delta U=0$ on $l_{3}$.

For any $v\in \mathcal{W}_{l_{2}}(\Sigma\cap B)$, using Green's formula on the bounded domain $\Sigma\cap B$ we get
$$
\begin{array}{rl}\vspace{1ex}
\displaystyle\int_{\Sigma\cap B}v\overline{g}\,dx= & \displaystyle\int_{\Sigma\cap B}v\overline{(\mathcal{L}^{\ast}_{A^{(1)},q^{(1)}}U)}\,dx \\ \vspace{1ex}
=& \displaystyle\int_{\Sigma\cap B}\left(\mathcal{L}_{A^{(1)},q^{(1)}}v\right)\overline{U}\,dx+i\int_{\partial(\Sigma\cap B)}\nu(x)\cdot Av \overline{U}\,dS\\ \vspace{1ex}
& +\displaystyle\int_{\partial(\Sigma\cap B)}\partial_{\nu}(-\Delta v)\overline{U}\,dS-\displaystyle\int_{\partial(\Sigma\cap B)}(-\Delta v)\overline{\partial_{\nu}U}\,dS\\ \vspace{1ex}
& +\displaystyle\int_{\partial(\Sigma\cap B)}\partial_{\nu}v\overline{(-\Delta U)}\,dS-\displaystyle\int_{\partial(\Sigma\cap B)}v\overline{(\partial_{\nu}(-\Delta U))}\,dS\\ \vspace{1ex}
=& 0.\\
\end{array}
$$
\end{proof}


Combining \eqref{identity2} with Proposition \ref{runge} we conclude
\begin{prop}
\begin{equation} \label{identity3}
\displaystyle\int_{\Sigma\cap B}((A^{(1)}-A^{(2)})\cdot Du_{1})\bar{v}\,dx+\int_{\Sigma\cap B}(q^{(1)}-q^{(2)})u_{1}\bar{v}\,dx=0.
\end{equation}
for all $u_{1}\in \mathcal{W}_{l_{2}}(\Sigma\cap B)$ and $v\in \mathcal{V}_{l_{1}}(\Sigma\cap B)$.
\end{prop}

\section{Construction of CGO solutions in the infinite slab}

In this section we construct CGO solutions $u_{1}\in \mathcal{W}_{l_{2}}(\Sigma\cap B)$ and $v\in \mathcal{V}_{l_{1}}(\Sigma\cap B)$. Let $\xi,\mu^{(1)},\mu^{(2)}\in\mathbb{R}^{n}$ be such that $|\mu^{(1)}|=|\mu^{(2)}|=1$ and $\mu^{(1)}\cdot\mu^{(2)}=\mu^{(1)}\cdot\xi=\mu^{(2)}\cdot\xi=0$. We set
\begin{equation}
\zeta_{1}:=\displaystyle\frac{ih\xi}{2}+i\sqrt{1-h^{2}\frac{|\xi|^{2}}{4}}\mu^{(1)}+\mu^{(2)}, \quad \zeta_{2}:=-\displaystyle\frac{ih\xi}{2}+i\sqrt{1-h^{2}\frac{|\xi|^{2}}{4}}\mu^{(1)}-\mu^{(2)}.
\end{equation}
Note that $\zeta_{1}\cdot\zeta_{1}=\zeta_{2}\cdot\zeta_{2}=0$, and $(\zeta_{1}+\overline{\zeta_{2}})/h=i\xi$. Here $h>0$ is a small semiclassical parameter. Note also that
\begin{equation}\label{zeta}
\zeta_{1}=i\mu^{(1)}+\mu^{(2)}+\mathcal{O}(h) \textrm{ and } \zeta_{2}=i\mu^{(1)}-\mu^{(2)}+\mathcal{O}(h) \textrm{ as } h\rightarrow 0,
\end{equation}
so $\zeta^{(0)}_{1}=i\mu^{(1)}+\mu^{(2)}$, $\zeta^{(0)}_{2}=i\mu^{(1)}-\mu^{(2)}$.

We first construct $u_{1}\in \mathcal{W}_{l_{2}}(\Sigma\cap B)$. To satisfy the condition $u_{1}|_{l_{2}}=\Delta u_{1}|_{l_{2}}=0$, we reflect $\Sigma\cap B$ with respect to the plane $x_{n}=0$ and denote this reflection by $(\Sigma\cap B)^{\ast}_{0}:=\{(x',-x_{n}):x=(x',x_{n})\in\Sigma\cap B\}$ where $x'=(x_{1},\cdots,x_{n-1})$. We extend the coefficients $A^{(1)}$ and $q^{(1)}$ to $(\Sigma\cap B)^{\ast}_{0}$ as follows: for the components $A^{(1)}_{j},j=1,\cdots,n-1$ and $q^{(1)}$, we extend them as even functions with respect to $x_{n}=0$, for $A^{(1)}_{n}$ we extend it as an odd function with respect to $x_{n}=0$, i.e. we set
$$
\begin{array}{rl} \vspace{1ex}
\tilde{A}^{(1)}_{j}(x)=&
\left\{
\begin{array}{ll} \vspace{1ex}
A^{(1)}_{j}(x',x_{n}) & 0<x_{n}<L \\
A^{(1)}_{j}(x',-x_{n}) & -L<x_{n}<0 \\
\end{array}
\right.
, \quad j=1,\cdots,n-1 \\ \vspace{1ex}

\tilde{A}^{(1)}_{n}(x)=&
\left\{
\begin{array}{ll} \vspace{1ex}
A^{(1)}_{n}(x',x_{n}) & 0<x_{n}<L \\
-A^{(1)}_{n}(x',-x_{n}) & -L<x_{n}<0 \\
\end{array}
\right. \\ \vspace{1ex}

\tilde{q}^{(1)}(x)=&
\left\{
\begin{array}{ll} \vspace{1ex}
q^{(1)}(x',x_{n}) & 0<x_{n}<L \\
q^{(1)}(x',-x_{n}) & -L<x_{n}<0 \\
\end{array}
\right. .\\
\end{array}
$$

For the moment, let us assume $A^{(1)}_{n}|_{x_{n}=0}=0$ so that $\tilde{A}^{(1)}\in W^{1,\infty}((\Sigma\cap B)\cup(\Sigma\cap B)^{\ast}_{0})$ and $\tilde{q}^{(1)}\in L^{\infty}((\Sigma\cap B)\cup(\Sigma\cap B)^{\ast}_{0})$. We will come back to the general case after establishing Proposition \ref{curlvanish}.

Proposition \ref{existence} implies that there exist CGO solutions of the form
$$\tilde{u}_{1}(x,\zeta_{1},h)=e^{x\cdot\zeta_{1}/h}(a_{1}(x,\zeta^{(0)}_{1})+r_{1}(x,\zeta_{1},h))\in H^{2}((\Sigma\cap B)\cup(\Sigma\cap B)^{\ast}_{0})$$
which satisfy the equation $\mathcal{L}_{\tilde{A}^{(1)},\tilde{q}^{(1)}}\tilde{u}_{1}=0$ in the bounded region $(\Sigma\cap B)\cup(\Sigma\cap B)^{\ast}_{0}$ with
\begin{equation}
((i\mu^{(1)}+\mu^{(2)})\cdot\nabla)^{2}a_{1}=0 \quad \textrm{ in } (\Sigma\cap B)\cup(\Sigma\cap B)^{\ast}_{0},
\end{equation}
\begin{equation}
\|r_{1}\|_{H^{1}_{\textrm{scl}}((\Sigma\cap B)\cup(\Sigma\cap B)^{\ast}_{0})}=\mathcal{O}(h).
\end{equation}
Let
\begin{equation} \label{CGO_u}
u_{1}(x):=\tilde{u}_{1}(x',x_{n})-\tilde{u}_{1}(x',-x_{n}) \quad\quad x\in\Sigma\cap B.
\end{equation}
Then it is easy to check that $u_{1}\in \mathcal{W}_{l_{2}}(\Sigma\cap B)$.\\

To construct $v\in \mathcal{V}_{l_{1}}(\Sigma\cap B)$, we notice that $\mathcal{L}^{\ast}_{A,q}=\mathcal{L}_{\overline{A},i^{-1}\nabla\cdot\overline{A}+\overline{q}}$, so $\mathcal{L}^{\ast}_{A^{(2)},q^{(2)}}v=0$ is equivalent to $\mathcal{L}_{A^{(3)},q^{(3)}}v=0$ where $A^{(3)}:=\overline{A^{(2)}}$ and $q^{(3)}:=i^{-1}\nabla\cdot\overline{A^{(2)}}+\overline{q^{(2)}}$. In the following we will construct $v$ such that $\mathcal{L}_{A^{(3)},q^{(3)}}v=0$ with $v|_{l_{1}}=\Delta v|_{l_{1}}=0$. To this end, we reflect $\Sigma\cap B$ with respect to the plane $x_{n}=L$ and denote this reflection by $(\Sigma\cap B)^{\ast}_{L}:=\{(x',-x_{n}+2L):x=(x',x_{n})\in\Sigma\cap B\}$. We extend the coefficients $A^{(3)}$ and $q^{(3)}$ to $(\Sigma\cap B)^{\ast}_{L}$ as follows: for $A^{(3)}_{j},j=1,\cdots,n-1$ and $q^{(3)}$ we extend them as even functions with respect to $x_{n}=L$, for $A^{(3)}_{n}$ we extend it as an odd function with respect to $x_{n}=L$, i.e.
$$
\begin{array}{rl} \vspace{1ex}
\tilde{A}^{(3)}_{j}(x)=&
\left\{
\begin{array}{ll} \vspace{1ex}
A^{(3)}_{j}(x',x_{n}) & 0<x_{n}<L \\
A^{(3)}_{j}(x',-x_{n}+2L) & L<x_{n}<2L \\
\end{array}
\right.
, \quad j=1,\cdots,n-1 \\ \vspace{1ex}

\tilde{A}^{(3)}_{n}(x)=&
\left\{
\begin{array}{ll} \vspace{1ex}
A^{(3)}_{n}(x',x_{n}) & 0<x_{n}<L \\
-A^{(3)}_{n}(x',-x_{n}+2L) & L<x_{n}<2L \\
\end{array}
\right. \\ \vspace{1ex}

\tilde{q}^{(3)}(x)=&
\left\{
\begin{array}{ll} \vspace{1ex}
q^{(3)}(x',x_{n}) & 0<x_{n}<L \\
q^{(3)}(x',-x_{n}+2L) & L<x_{n}<2L \\
\end{array}
\right. .\\
\end{array}
$$

Again, first we assume $A^{(2)}_{n}|_{x_{n}=L}=0$ so that $A^{(3)}_{n}|_{x_{n}=L}=0$, $\tilde{A}^{(3)}\in W^{1,\infty}((\Sigma\cap B)\cup(\Sigma\cap B)^{\ast}_{L})$ and $\tilde{q}^{(3)}\in L^{\infty}((\Sigma\cap B)\cup(\Sigma\cap B)^{\ast}_{L})$. The general case will be dealt with below Proposition \ref{curlvanish}.

Proposition \ref{existence} implies that there exist CGO solutions of the form
$$\tilde{v}(x,\zeta_{2},h)=e^{x\cdot\zeta_{2}/h}(a_{2}(x,\zeta^{(0)}_{2})+r_{2}(x,\zeta_{2},h))\in H^{2}((\Sigma\cap B)\cup(\Sigma\cap B)^{\ast}_{L})$$
which satisfy the equation $\mathcal{L}_{\tilde{A}^{(3)},\tilde{q}^{(3)}}\tilde{v}=0$ in the bounded region $(\Sigma\cap B)\cup(\Sigma\cap B)^{\ast}_{L}$ with
$$
((i\mu^{(1)}-\mu^{(2)})\cdot\nabla)^{2}a_{2}=0 \quad \textrm{ in } (\Sigma\cap B)\cup(\Sigma\cap B)^{\ast}_{L},
$$
$$
\|r_{2}\|_{H^{1}_{\textrm{scl}}((\Sigma\cap B)\cup(\Sigma\cap B)^{\ast}_{L})}=\mathcal{O}(h).
$$
Let
\begin{equation}\label{CGO_v}
v(x):=\tilde{v}(x',x_{n})-\tilde{v}(x',-x_{n}+2L) \quad\quad x\in\Sigma\cap B.
\end{equation}
Then it is easy to check that $v\in \mathcal{V}_{l_{1}}(\Sigma\cap B)$.

We write down the CGO solutions \eqref{CGO_u} and \eqref{CGO_v} explicitly for future references:
\begin{equation}\label{CGO1}
u_{1}(x)=e^{x\cdot\zeta_{1}/h}(a_{1}(x)+r_{1}(x))-e^{(x',-x_{n})\cdot\zeta_{1}/h}(a_{1}(x',-x_{n})+r_{1}(x',-x_{n}))
\end{equation}
\begin{equation}\label{CGO2}
v(x)=e^{x\cdot\zeta_{2}/h}(a_{2}(x)+r_{2}(x))-e^{(x',-x_{n}+2L)\cdot\zeta_{2}/h}(a_{2}(x',-x_{n}+2L)+r_{2}(x',-x_{n}+2L))
\end{equation}
where $a_{1}\in C^{\infty}(\overline{(\Sigma\cap B)\cup(\Sigma\cap B)^{\ast}_{0}})$ ,$a_{2}\in C^{\infty}(\overline{(\Sigma\cap B)\cup(\Sigma\cap B)^{\ast}_{L}})$ and
\begin{equation}\label{condition1}
((i\mu^{(1)}+\mu^{(2)})\cdot\nabla)^{2} a_{1}=0 \quad \textrm{ in } (\Sigma\cap B)\cup(\Sigma\cap B)^{\ast}_{0},
\end{equation}
\begin{equation}\label{condition2}
((i\mu^{(1)}-\mu^{(2)})\cdot\nabla)^{2} a_{2}=0 \quad \textrm{ in } (\Sigma\cap B)\cup(\Sigma\cap B)^{\ast}_{L},
\end{equation}
\begin{equation}\label{condition3}
\|r_{1}\|_{H^{1}_{\textrm{scl}}((\Sigma\cap B)\cup(\Sigma\cap B)^{\ast}_{0})}=\mathcal{O}(h),
\end{equation}
\begin{equation}\label{condition4}
\|r_{2}\|_{H^{1}_{\textrm{scl}}((\Sigma\cap B)\cup(\Sigma\cap B)^{\ast}_{L})}=\mathcal{O}(h).
\end{equation}

\section{Proof of Theorem 1.1}

We are ready to prove our first main theorem. We will substitute the CGO solutions constructed in last section into \eqref{identity3}. To this end we compute
\begin{equation}\label{exp1}
\begin{array}{rl} \vspace{1ex}
e^{x\cdot\zeta_{1}/h}e^{x\cdot\overline{\zeta}_{2}/h}=&e^{ix\cdot\xi} \\ \vspace{1ex}
e^{(x',-x_{n})\cdot\zeta_{1}/h}e^{x\cdot\overline{\zeta}_{2}/h}=&e^{-2\mu^{(2)}_{n}x_{n}/h+ib_{1}} \\ \vspace{1ex}
e^{x\cdot\zeta_{1}/h}e^{(x',-x_{n}+2L)\cdot\overline{\zeta}_{2}/h}=&e^{2\mu^{(2)}_{n}(x_{n}-L)/h+ib_{2}} \\ \vspace{1ex}
e^{(x',-x_{n})\cdot\zeta_{1}/h}e^{(x',-x_{n}+2L)\cdot\overline{\zeta}_{2}/h}=&e^{-2L\mu^{(2)}_{n}/h+ib_{3}} \\
\end{array}
\end{equation}
where $b_{1}, b_{2}, b_{3}\in\mathbb{R}^{n}$ are defined by
$$
\begin{array}{rl} \vspace{1ex}
b_{1}:=&x'\cdot\xi'-\displaystyle\frac{2}{h}\sqrt{1-h^{2}\frac{|\xi|^{2}}{4}}\mu^{(1)}_{n}x_{n}, \\ \vspace{1ex}
b_{2}:=&x'\cdot\xi'+\displaystyle\frac{2}{h}\sqrt{1-h^{2}\frac{|\xi|^{2}}{4}}\mu^{(1)}_{n}(x_{n}-L)+L\xi_{n}, \\ \vspace{1ex}
b_{3}:=&x'\cdot\xi'-\displaystyle\frac{2L}{h}\sqrt{1-h^{2}\frac{|\xi|^{2}}{4}}\mu^{(1)}_{n}-x_{n}\xi_{n}+L\xi_{n}. \\
\end{array}
$$
We further assume that $\mu^{(2)}_{n}>0$, hence for $0<x_{n}<L$ the following pointwise convergence holds as $h\rightarrow 0+$:
\begin{equation}\label{vanish}
\begin{array}{rl} \vspace{1ex}
|e^{(x',-x_{n})\cdot\zeta_{1}/h}e^{x\cdot\overline{\zeta}_{2}/h}|&\rightarrow 0 \textrm{ as } h\rightarrow 0+,\\ \vspace{1ex}
|e^{x\cdot\zeta_{1}/h}e^{(x',-x_{n}+2L)\cdot\overline{\zeta}_{2}/h}|&\rightarrow 0 \textrm{ as } h\rightarrow 0+,\\ \vspace{1ex}
|e^{(x',-x_{n})\cdot\zeta_{1}/h}e^{(x',-x_{n}+2L)\cdot\overline{\zeta}_{2}/h}|&\rightarrow 0 \textrm{ as } h\rightarrow 0+.\\
\end{array}
\end{equation}
Therefore, with the CGO solutions $u_{1}$ and $v$ given by \eqref{CGO1} and \eqref{CGO2}, we conclude from \eqref{condition3} \eqref{condition4} and \eqref{vanish} that
\begin{equation}\label{term2}
h\displaystyle\int_{\Sigma\cap B}(q^{(1)}-q^{(2)})u_{1}\bar{v}\,dx\rightarrow 0 \quad \textrm{ as } h\rightarrow 0+.
\end{equation}

On the other hand, denote $\zeta^{\ast}_{j}=(\zeta'_{j},-(\zeta_{j})_{n})$ for $\zeta_{j}=(\zeta'_{j},(\zeta_{j})_{n})$, $j=1,2$. Using \eqref{CGO1} we compute
\begin{equation}
\begin{array}{rl} \vspace{1ex}
Du_{1}(x)=&-\displaystyle\frac{i\zeta_{1}}{h}e^{x\cdot\zeta_{1}/h}(a_{1}(x)+r_{1}(x))+e^{x\cdot\zeta_{1}/h}(Da_{1}(x)+Dr_{1}(x))\\ \vspace{1ex}
&+\displaystyle\frac{i\zeta^{\ast}_{1}}{h}e^{(x',-x_{n})\cdot\zeta_{1}/h}(a_{1}(x',-x_{n})+r_{1}(x',-x_{n})) \\ \vspace{1ex}
&-e^{(x',-x_{n})\cdot\zeta_{1}/h}(Da_{1}(x',-x_{n})+Dr_{1}(x',-x_{n})).\\
\end{array}
\end{equation}
Therefore, with the CGO solutions $u_{1}$ and $v$ given by \eqref{CGO1} and \eqref{CGO2}, we have from \eqref{zeta} \eqref{condition3} \eqref{condition4} \eqref{vanish} and the dominant convergence theorem that
\begin{equation}\label{term1}
\begin{array}{rl}\vspace{1ex}
 &h\displaystyle\int_{\Sigma\cap B}(A^{(1)}-A^{(2)})\cdot Du_{1}\bar{v}\,dx \\ \vspace{1ex}
 \rightarrow& (\mu^{(1)}-i\mu^{(2)})\cdot\displaystyle\int_{\Sigma\cap B}(A^{(1)}-A^{(2)}) e^{ix\cdot\xi}a_{1}\overline{a_{2}}\,dx \quad \textrm{ as } h\rightarrow 0+.\\
\end{array}
\end{equation}
Multiplying \eqref{identity3} by $h$ and letting $h\rightarrow 0+$ for the constructed solutions $u_{1}$ and $v$, we obtain from \eqref{term2} and \eqref{term1} that
\begin{equation}\label{identity_A1}
(\mu^{(1)}-i\mu^{(2)})\cdot\displaystyle\int_{\Sigma\cap B}(A^{(1)}-A^{(2)}) e^{ix\cdot\xi}a_{1}\overline{a_{2}}\,dx=0.
\end{equation}
This identity holds for all $a_{1}$ satisfying \eqref{condition1}, $a_{2}$ satisfying \eqref{condition2}, and for all $\mu^{(1)},\mu^{(2)},\xi\in\mathbb{R}^{n}$ such that $\mu^{(1)}\cdot\mu^{(2)}=\mu^{(1)}\cdot\xi=\mu^{(2)}\cdot\xi=0$ and $\mu^{(2)}_{n}>0$. Replace $\mu^{(1)}$ by $-\mu^{(1)}$ and subtract to find
\begin{equation}\label{identity_A2}
\mu^{(1)}\cdot\displaystyle\int_{\Sigma\cap B}(A^{(1)}-A^{(2)}) e^{ix\cdot\xi}a_{1}\overline{a_{2}}\,dx=0.
\end{equation}

\begin{prop} \label{curlvanish}
\begin{equation} \label{curl1}
\partial_{j}(A^{(1)}_{k}-A^{(2)}_{k})-\partial_{k}(A^{(1)}_{j}-A^{(2)}_{j})=0 \textrm{ in } \Sigma\cap B, \quad 1\leq j,k\leq n.
\end{equation}
\end{prop}
\begin{proof}
Obviously $a_{1}=a_{2}=1$ satisfies \eqref{condition1} and \eqref{condition2}. Inserting $a_{1}=a_{2}=1$ in \eqref{identity_A2} we get
\begin{equation}\label{fourier}
\mu^{(1)}\cdot(\widehat{A^{(1)}\chi_{\Sigma\cap B}}(\xi)-\widehat{A^{(2)}\chi_{\Sigma\cap B}}(\xi))=0
\end{equation}
where $\chi_{\Sigma\cap B}$ stands for the characteristic function of the set $\Sigma\cap B$ and $\widehat{A^{(j)}\chi_{\Sigma\cap B}}$ denotes the Fourier transform of $A^{(j)}\chi_{\Sigma\cap B}$.

To show the proposition, it suffices to consider the case when $j\neq k$. Let $e_{1},\cdots,e_{n}$ be the standard orthonormal basis in $\mathbb{R}^{n}$. Let $\xi=(\xi_{1},\cdots,\xi_{n})$ with $\xi_{j}>0, j=1,\cdots,n$. Define
$$\mu^{(1)}=-\xi_{k}e_{j}+\xi_{j}e_{k} \quad\quad 1\leq j,k\leq n, j\neq k.$$
To define $\mu^{(2)}$ we consider two cases: if $j,k$ are such that $1\leq j,k<n$, define
$$\mu^{(2)}=-\xi_{j}\xi_{n}e_{j}-\xi_{k}\xi_{n}e_{k}+(\xi^{2}_{j}+\xi^{2}_{k})e_{n};$$
if $k=n$ and $j$ is such that $1\leq j<n$, define
$$\mu^{(2)}=(-\xi^{2}_{j}-\xi^{2}_{n})e_{l}+\xi_{l}\xi_{j}e_{j}+\xi_{l}\xi_{n}e_{n}$$
with some $l\neq j,n$, which exists since $n\geq 3$. In either case it is easy to check $\mu^{(1)}\cdot\mu^{(2)}=\mu^{(1)}\cdot\xi=\mu^{(2)}\cdot\xi=0$ and $\mu^{(2)}_{n}>0$. For such $\mu^{(1)}$ and $\xi$ we get from \eqref{fourier} that
$$\xi_{j}\cdot(\widehat{A^{(2)}_{k}\chi_{\Sigma\cap B}}(\xi)-\widehat{A^{(1)}_{k}\chi_{\Sigma\cap B}}(\xi))-\xi_{k}\cdot(\widehat{A^{(2)}_{j}\chi_{\Sigma\cap B}}(\xi)-\widehat{A^{(1)}_{j}\chi_{\Sigma\cap B}}(\xi))=0,$$
$1\leq j,k\leq n, j\neq k$ for all $\xi\in\mathbb{R}^{n},\xi_{1}>0,\cdots,\xi_{n}>0$, and thus everywhere by analyticity of the Fourier transform. This completes the proof.
\end{proof}

By Proposition \ref{curlvanish}, we conclude $d A^{(1)}=d A^{(2)}$ in $\Sigma$. As $\Sigma$ is simply connected, there exists a compactly supported $\Phi\in C^{1,1}(\overline{\Sigma})$ such that
$$A^{(1)}-A^{(2)}=\nabla\Phi \quad\quad \textrm{ in } \Sigma.$$
In particular, $\Phi=0$ along $\partial B\cap\Sigma$.

Recall that in the construction of the CGO solutions above, we have assumed that $A^{(1)}_{n}|_{x_{n}=0}=0$ and $A^{(2)}_{n}|_{x_{n}=L}=0$. Now we show why our results are independent of such assumptions. Indeed, for $A^{(1)}$, there exists $\Psi^{(1)}\in C^{1,1}(\overline{\Sigma})$ with compact support such that $\Psi^{(1)}|_{\partial\Sigma}=0$ and $\partial_{\nu}\Psi^{(1)}=-A^{(1)}\cdot\nu$ on $\partial\Sigma$, where as usual $\nu$ is the unit outer normal vector on $\partial\Sigma$. Then $A^{(1)}+\nabla\Psi^{(1)}$ satisfies $(A^{(1)}_{n}+\nabla\Psi^{(1)}_{n})|_{x_{n}=0}=0$. See \cite[Theorem 1.3.3]{H} for the existence of $\Psi^{(1)}$. Similarly, we can find $\Psi^{(2)}\in C^{1,1}(\overline{\Sigma})$ with compact support such that $\Psi^{(2)}|_{\partial\Sigma}=0$ and $\partial_{\nu}\Psi^{(2)}=-A^{(2)}\cdot\nu$ on $\partial\Sigma$. Then $(A^{(2)}_{n}+\nabla\Psi^{(2)}_{n})|_{x_{n}=L}=0$. Therefore, we may replace $A^{(j)}$ by $A^{(j)}+\nabla\Psi^{(j)},j=1,2$ to fulfill the assumption. After the replacement, Proposition \ref{curlvanish} will give $d(A^{(1)}+\nabla\Psi^{(1)})=d(A^{(2)}+\nabla\Psi^{(2)})$ in $\Sigma$. As above we can find a compactly supported function $\Phi'\in C^{1,1}(\overline{\Sigma})$ such that
$$A^{(1)}+\nabla\Psi^{(1)}-A^{(2)}-\nabla\Psi^{(2)}=\nabla\Phi' \quad\quad \textrm{ in } \Sigma.$$
Define $\Phi:=\Phi'-\Psi^{(1)}+\Psi^{(2)}$, then $\Phi\in C^{1,1}(\overline{\Sigma})$ is compactly supported and satisfies $A^{(1)}-A^{(2)}=\nabla\Phi$. In particular $\Phi=0$ on $\partial B\cap\Sigma$. We are back to the same situation.

Next, we establish a proposition which asserts that $\Phi=0$ on $\Gamma_{1}\cup\Gamma_{2}$.
\begin{prop}
$\Phi=0$ along $\partial(\Sigma\cap B)$.
\end{prop}
\begin{proof}
Notice \eqref{condition1} implies that in the expression \eqref{CGO1}, we may replace $a_{1}$ by $g_{1}a_{1}$ if $g_{1}\in C^{\infty}(\overline{(\Sigma\cap B)\cup(\Sigma\cap B)^{\ast}_{0}\cup(\Sigma\cap B)^{\ast}_{L}})$ satisfies
$$(i\mu^{(1)}+\mu^{(2)})\cdot\nabla g_{1}=0 \quad\quad \textrm{ in } (\Sigma\cap B)\cup(\Sigma\cap B)^{\ast}_{0}\cup(\Sigma\cap B)^{\ast}_{L}.$$
Thus \eqref{identity_A1} becomes
$$(\mu^{(1)}-i\mu^{(2)})\cdot\displaystyle\int_{\Sigma\cap B}\nabla\Phi g_{1}e^{ix\cdot\xi}a_{1}\overline{a_{2}}\,dx=0.$$
Set $\xi=0$, $a_{1}=a_{2}=1$ and multiply by $i$:
\begin{equation}\label{g}
(i\mu^{(1)}+\mu^{(2)})\cdot\displaystyle\int_{\Sigma\cap B}\nabla\Phi g_{1}\,dx=0
\end{equation}
As $\mu^{(1)}\cdot\mu^{(2)}=0$ and $|\mu^{(1)}|=|\mu^{(1)}|=1$, we can make a change of variable so that $(i\mu^{(1)}+\mu^{(2)})\cdot\nabla$ becomes a $\overline{\partial}$-operator as follows. Complete the set $\{\mu^{(2)}$, $\mu^{(1)}\}$ to an orthonormal basis in $\mathbb{R}^{n}$, say $\{\mu^{(2)},\mu^{(1)},\mu^{(3)},\cdots,\mu^{(n)}\}$; introduce new coordinates $y=(y_{1},\cdots,y_{n})\in\mathbb{R}^{n}$ with respect to this orthonormal basis by defining $y_{1}=x\cdot\mu^{(2)},y_{2}=x\cdot\mu^{(1)},y_{j}=x\cdot\mu^{(j)},j=3,\cdots,n$; in other words, we made an orthogonal transformation $T:\mathbb{R}^{n}\rightarrow\mathbb{R}^{n}$, $T(x)=y$. Denote $z=y_{1}+iy_{2}$ and $\partial_{\overline{z}}=\frac{1}{2}(\partial_{y_{1}}+i\partial_{y_{2}})$. Then $(i\mu^{(1)}+\mu^{(2)})\cdot\nabla=2\partial_{\overline{z}}$, and in the new coordinates \eqref{g} becomes
$$\displaystyle\int_{T(\Sigma\cap B)}g_{1}\partial_{\overline{z}}\Phi\,dy=0$$
for all $g_{1}\in C^{\infty}(\overline{(\Sigma\cap B)\cup(\Sigma\cap B)^{\ast}_{0}\cup(\Sigma\cap B)^{\ast}_{L}})$ satisfying $\partial_{\overline{z}}g_{1}=0$. Replacing $\mu^{(1)}$ by $-\mu^{(1)}$, in the same way we can show
$$\displaystyle\int_{T(\Sigma\cap B)}g_{2}\partial_{z}\Phi\,dy=0$$
for all $g_{2}\in C^{\infty}(\overline{(\Sigma\cap B)\cup(\Sigma\cap B)^{\ast}_{L}\cup(\Sigma\cap B)^{\ast}_{0}})$ satisfying $\partial_{z}g_{2}=0$. Taking $g_{j}(y)=g'_{j}(z)\otimes g''_{j}(y''),\; j=1,2, \; y''=(y_{3},\cdots,y_{n})$ and varying $g''_{j}$ yields
$$\displaystyle\int_{T_{y''}}g'_{1}(z)\partial_{\overline{z}}\Phi\,dz\wedge d\overline{z}=0 \quad\quad \displaystyle\int_{T_{y''}}g'_{2}(\overline{z})\partial_{z}\Phi\,dz\wedge d\overline{z}=0$$
Here $T_{y''}$ is the intersection of $T(\Sigma\cap B)$ with the two dimensional plane $\{(y_{1},y_{2},y'')\in\mathbb{R}^{n}:y'' \textrm{ fixed}\}$, $\partial_{\overline{z}}g'_{1}=0$ and $\partial_{z}g'_{2}=0$. Notice that $\partial T_{y''}$ is piecewise smooth. Since
$$d(g'_{1}(z)\Phi\,dz)=g'_{1}(z)\partial_{\overline{z}}\Phi\,d\overline{z}\wedge dz, \quad\quad d(g'_{2}(\overline{z})\Phi\,d\overline{z})=g'_{2}(\overline{z})\partial_{z}\Phi\,dz\wedge d\overline{z},$$
we obtain from Stokes' formula that
$$\displaystyle\int_{\partial T_{y''}}g'_{1}(z)\Phi\,dz=0 \quad\quad \displaystyle\int_{\partial T_{y''}}g'_{2}(\overline{z})\Phi\,d\overline{z}=0.$$
Taking $g'_{2}=\overline{g'_{1}}$ we see that
$$\displaystyle\int_{\partial T_{y''}}g'_{1}(z)\Phi\,dz=0 \quad\quad \displaystyle\int_{\partial T_{y''}}g'_{1}(z)\overline{\Phi}\,dz=0$$
Hence
$$\displaystyle\int_{\partial T_{y''}}g'_{1}(z)\textrm{Re }\Phi\,dz=\displaystyle\int_{\partial T_{y''}}g'_{1}(z)\textrm{Im }\Phi\,dz=0.$$
for all holomorphic functions $g'_{1}\in C^{\infty}(\overline{T_{y''}})$. Arguing as in \cite[Lemma 5.1]{FKSU}, we can find holomorphic functions $F_{1},F_{2}\in C(\overline{T_{y''}})$ such that
$$F_{1}|_{\partial T_{y''}}=\textrm{Re }\Phi|_{\partial T_{y''}} \quad\quad F_{2}|_{\partial T_{y''}}=\textrm{Im }\Phi|_{\partial T_{y''}}.$$
Moreover, $\Delta\textrm{Im }F_{j}=0$ in $T_{y''}$ and $\textrm{Im }F_{j}|_{\partial T_{y''}}=0$. Thus, $F_{j},\; j=1,2,$ are real-valued and thus constant on $T_{y''}$. Therefore, $\Phi$ is constant along $\partial T_{y''}$. In the $x$-coordinate system, we see that the function $\Phi(x)$ is constant on the boundary of the intersection $T^{-1}(\Pi_{y''})\cap (\Sigma\cap B)$ for all $y''\in\mathbb{R}^{n-2}$, where $T^{-1}(\Pi_{y''})$ is defined by
$$T^{-1}(\Pi_{y''}):=\{x=y_{1}\mu^{(2)}+y_{2}\mu^{(1)}+\sum^{n}_{j=3}y_{j}\mu^{(j)}: y_{1},y_{2}\in\mathbb{R},y''=(y_{3},\cdots,y_{n})\}.$$
Setting $\mu^{(1)}=e_{j},\; j=1,\cdots,n-1$ and $\mu^{(2)}=e_{n}$, then varying $y''$ gives that $\Phi$ vanishes on $\partial(\Sigma\cap B)$. This completes the proof.
\end{proof}

To show that $A^{(1)}=A^{(2)}$ consider \eqref{identity_A1} with $a_{2}=1$ and $a_{1}$ satisfying
$$((i\mu^{(1)}+\mu^{(2)})\cdot\nabla) a_{1}=1 \quad \textrm{ in } (\Sigma\cap B)\cup(\Sigma\cap B)^{\ast}_{0}.$$
This choice is possible thanks to \eqref{condition1}. We have from \eqref{identity_A1} that
$$(i\mu^{(1)}+\mu^{(2)})\cdot\displaystyle\int_{\Sigma\cap B}(\nabla\Phi)e^{ix\cdot\xi}a_{1}\,dx=0$$
Integrating by parts and using the fact that $\Phi=0$ along $\partial(\Sigma\cap B)$ and $\mu^{(1)}\cdot\xi=\mu^{(2)}\cdot\xi=0$ we obtain
$$0=\displaystyle\int_{\Sigma\cap B}\Phi(x)e^{ix\cdot\xi}[(i\mu^{(1)}+\mu^{(2)})\cdot\nabla a_{1}]\,dx=\int_{\Sigma\cap B}\Phi(x)e^{ix\cdot\xi}\,dx.$$
This indicates that Fourier transform of the function $\Phi\chi_{\Sigma\cap B}$ vanishes. Thus $\Phi=0$ in $\Sigma\cap B$, and therefore $A^{(1)}=A^{(2)}$.\vspace{1ex}

Inserting $A^{(1)}=A^{(2)}$ in \eqref{identity3} gives
$$\displaystyle\int_{\Sigma\cap B}(q^{(1)}-q^{(2)})u_{1}\bar{v}\,dx=0.$$
Let $u_{1}$ and $v$ be the CGO solutions given by \eqref{CGO1} and \eqref{CGO2}. Taking the limit $h\rightarrow 0+$, from \eqref{condition3} \eqref{condition4} \eqref{vanish} we get
$$\displaystyle\int_{\Sigma\cap B}(q^{(1)}-q^{(2)})e^{ix\cdot\xi}a_{1}\bar{a}_{2}\,dx=0$$
where $a_{1}$ and $a_{2}$ satisfy \eqref{condition1} and \eqref{condition2} respectively. In particular, for $a_{1}=a_{2}=1$ this identity becomes
\begin{equation}\label{identity_q}
\displaystyle\int_{\Sigma\cap B}(q^{(1)}-q^{(2)})e^{ix\cdot\xi}\,dx=0
\end{equation}
for all $\xi$ such that there exist $\mu^{(1)},\mu^{(2)}\in\mathbb{R}^{n}$ such that
$$\mu^{(1)}\cdot\mu^{(2)}=\xi\cdot\mu^{(1)}=\xi\cdot\mu^{(2)}=0, \quad |\mu^{(1)}|=|\mu^{(2)}|=1,\quad \mu^{(2)}_{n}>0.$$

Write $\xi=(\xi',\xi_{n-1},\xi_{n})$ with $\xi'\in\mathbb{R}^{n-2}$. If $\xi_{n-1}\neq 0$, we can choose
$$\mu^{(2)}=\displaystyle\frac{1}{\sqrt{1+\frac{\xi^{2}_{n}}{\xi^{2}_{n-1}}}}(0_{\mathbb{R}^{n-2}},\frac{-\xi_{n}}{\xi_{n-1}},1),$$
which satisfies $\xi\cdot\mu^{(2)}=0$, $|\mu^{(2)}|=1$ and $\mu^{(2)}_{n}>0$. Since $n\geq 3$, we can find a third unit vector $\mu^{(1)}$ so that $\{\mu^{(1)},\mu^{(2)},\xi\}$ are mutually orthogonal. Thus \eqref{identity_q} indicates that $\widehat{q^{(1)}\chi_{\Sigma\cap B}}(\xi)=\widehat{q^{(2)}\chi_{\Sigma\cap B}}(\xi)$ for $\xi$ with $\xi_{n-1}\neq 0$, and therefore for all $\xi\in\mathbb{R}^{n}$ as both Fourier transforms are continuous functions. This completes the proof of Theorem 1.1.
\vspace{3ex}

\section{Proof of Theorem 1.2}
In this section we show Theorem 1.2. First, arguing as in the proof of Theorem 1.1, we derive identity \eqref{identity3} for all $u_{1}\in \mathcal{W}_{l_{2}}(\Sigma\cap B)$ and $v\in \mathcal{V}_{l_{2}}(\Sigma\cap B)$. Next, we construct CGO solutions to be used in the proof of Theorem 1.2. We have constructed $u_{1}\in \mathcal{W}_{l_{2}}(\Sigma\cap B)$ in \eqref{CGO_u}, now we construct $v\in \mathcal{V}_{l_{2}}(\Sigma\cap B)$. As in the construction of $u_{1}$, we will reflect the coefficients with respect to the plane $x_{n}=0$. Recall that we have introduced $A^{(3)}=\overline{A^{(2)}}$ and $q^{(3)}=i^{-1}\nabla\cdot\overline{A^{(2)}}+\overline{q^{(2)}}$ so that $\mathcal{L}^{\ast}_{A^{(2)},q^{(2)}}=\mathcal{L}_{A^{(3)},q^{(3)}}$. For $A^{(3)}_{j}, j=1,\cdots,n-1$ and $q^{(3)}$, we extend them as even functions with respect to $x_{n}=0$; for $A^{(3)}_{n}$, we extend it as an odd function with respect to $x_{n}=0$, i.e. we set
$$
\begin{array}{rl} \vspace{1ex}
\tilde{A}^{(3)}_{j}(x)=&
\left\{
\begin{array}{ll} \vspace{1ex}
A^{(3)}_{j}(x',x_{n}) & 0<x_{n}<L \\
A^{(3)}_{j}(x',-x_{n}) & -L<x_{n}<0 \\
\end{array}
\right.
, \quad j=1,\cdots,n-1 \\ \vspace{1ex}

\tilde{A}^{(3)}_{n}(x)=&
\left\{
\begin{array}{ll} \vspace{1ex}
A^{(3)}_{n}(x',x_{n}) & 0<x_{n}<L \\
-A^{(3)}_{n}(x',-x_{n}) & -L<x_{n}<0 \\
\end{array}
\right. \\ \vspace{1ex}

\tilde{q}^{(3)}(x)=&
\left\{
\begin{array}{ll} \vspace{1ex}
q^{(3)}(x',x_{n}) & 0<x_{n}<L \\
q^{(3)}(x',-x_{n}) & -L<x_{n}<0 \\
\end{array}
\right. .\\
\end{array}
$$

Without loss of generality we assume $A^{(3)}_{n}|_{x_{n}=0}=0$, as we did before. Then $\tilde{A}^{(3)}\in W^{1,\infty}((\Sigma\cap B)\cup(\Sigma\cap B)^{\ast}_{0})$ and $\tilde{q}^{(3)}\in L^{\infty}((\Sigma\cap B)\cup(\Sigma\cap B)^{\ast}_{0})$. Proposition \ref{existence} implies that there exist CGO solutions of the form
$$\tilde{v}(x,\zeta_{2},h)=e^{x\cdot\zeta_{2}/h}(a_{2}(x,\zeta^{(0)}_{2})+r_{2}(x,\zeta_{2},h))\in H^{2}((\Sigma\cap B)\cup(\Sigma\cap B)^{\ast}_{0})$$
which satisfy the equation $\mathcal{L}_{\tilde{A}^{(3)},\tilde{q}^{(3)}}v=0$ in the bounded region $(\Sigma\cap B)\cup(\Sigma\cap B)^{\ast}_{0}$ with
\begin{equation}
((i\mu^{(1)}-\mu^{(2)})\cdot\nabla)^{2}a_{2}=0 \quad \textrm{ in } (\Sigma\cap B)\cup(\Sigma\cap B)^{\ast}_{0},
\end{equation}
\begin{equation}
\|r_{2}\|_{H^{1}_{\textrm{scl}}((\Sigma\cap B)\cup(\Sigma\cap B)^{\ast}_{0})}=\mathcal{O}(h).
\end{equation}
Let
\begin{equation} \label{CGO_v2}
v(x):=\tilde{v}(x',x_{n})-\tilde{v}(x',-x_{n}) \quad\quad x\in\Sigma\cap B.
\end{equation}
Then it is easy to see that $v\in \mathcal{V}_{l_{2}}(\Sigma\cap B)$.\\

It will be convenient to write down the CGO solutions \eqref{CGO_v2} explicitly for future references:
\begin{equation}\label{CGO3}
v(x)=e^{x\cdot\zeta_{2}/h}(a_{2}(x)+r_{2}(x))-e^{(x',-x_{n})\cdot\zeta_{2}/h}(a_{2}(x',-x_{n})+r_{2}(x',-x_{n}))
\end{equation}
where $a_{2}\in C^{\infty}(\overline{(\Sigma\cap B)\cup(\Sigma\cap B)^{\ast}_{0}})$ and
\begin{equation}\label{condition5}
((i\mu^{(1)}-\mu^{(2)})\cdot\nabla)^{2} a_{2}=0 \quad \textrm{ in } (\Sigma\cap B)\cup(\Sigma\cap B)^{\ast}_{0},
\end{equation}
\begin{equation}\label{condition6}
\|r_{2}\|_{H^{1}_{\textrm{scl}}((\Sigma\cap B)\cup(\Sigma\cap B)^{\ast}_{0})}=\mathcal{O}(h).
\end{equation}

We will substitute the solutions \eqref{CGO1} and \eqref{CGO3} into \eqref{identity3}. To this end we compute
\begin{equation}\label{exp2}
\begin{array}{rl} \vspace{1ex}
e^{x\cdot\zeta_{1}/h}e^{x\cdot\overline{\zeta}_{2}/h}=&e^{ix\cdot\xi} \\ \vspace{1ex}
e^{x\cdot\zeta_{1}/h}e^{(x',-x_{n})\cdot\overline{\zeta}_{2}/h}=&e^{ix\cdot\xi_{+}+2\mu^{(2)}_{n}x_{n}/h}\\ \vspace{1ex}
e^{(x',-x_{n})\cdot\zeta_{1}/h}e^{x\cdot\overline{\zeta}_{2}/h}=&e^{ix\cdot\xi_{-}-2\mu^{(2)}_{n}x_{n}/h} \\ \vspace{1ex}
e^{(x',-x_{n})\cdot\zeta_{1}/h}e^{(x',-x_{n})\cdot\overline{\zeta}_{2}/h}=&e^{i(x',-x_{n})\cdot\xi} \\
\end{array}
\end{equation}
where
$$
\xi_{\pm}=\left(\xi',\pm\frac{2}{h}\sqrt{1-h^{2}\frac{|\xi|^{2}}{4}}\mu^{(1)}_{n}\right).
$$
Moreover, we assume $\mu^{(1)}_{n}\neq 0$ and $\mu^{(2)}_{n}=0$, so $\xi_{\pm}\rightarrow\infty$ as $h\rightarrow 0$. Then we have
\begin{equation}\label{vanish2}
\zeta_{1}\cdot\displaystyle\int_{\Sigma\cap B} (A^{(1)}-A^{(2)})e^{x\cdot\zeta_{1}/h}e^{(x',-x_{n})\cdot\overline{\zeta}_{2}/h}a_{1}a_{2}\,dx\rightarrow 0
\end{equation}
as $h\rightarrow 0$ by the Riemann-Lebesgue lemma. Similarly
\begin{equation}\label{vanish3}
\zeta_{1}\cdot\displaystyle\int_{\Sigma\cap B} (A^{(1)}-A^{(2)})e^{(x',-x_{n})\cdot\overline{\zeta}_{1}/h}e^{x\cdot\overline{\zeta}_{2}/h}a_{1}a_{2}\,dx\rightarrow 0
\end{equation}
as $h\rightarrow 0$. Therefore, multiplying \eqref{identity3} by $h$ and taking the limit $h\rightarrow 0$ gives
$$
\begin{array}{rl}\vspace{1ex}
& \left(\mu^{(1)}-i\mu^{(2)}\right)\cdot\displaystyle\int_{\Sigma\cap B}\left(A^{(1)}-A^{(2)}\right)e^{ix\cdot\xi}a_{1}\bar{a}_{2}\,dx\\ \vspace{1ex}
+&\left((\mu^{(1)})'-i(\mu^{(2)})',-(\mu^{(1)}-i\mu^{(2)})\right)\cdot\displaystyle\int_{\Sigma\cap B}\left(A^{(1)}-A^{(2)}\right)e^{i(x',-x_{n})\cdot\xi}\\ \vspace{1ex}
& a_{1}(x',-x_{n})\overline{a_{2}(x',-x_{n})}\,dx\rightarrow 0.
\end{array}
$$
Set $\tilde{A}^{(2)}=\overline{\tilde{A}^{(3)}}$. After a change of variable, this expression becomes
\begin{equation}\label{mu1}
\left(\mu^{(1)}-i\mu^{(2)}\right)\cdot\displaystyle\int_{(\Sigma\cap B)\cup(\Sigma\cap B)^{\ast}_{0}}\left(\tilde{A}^{(1)}-\tilde{A}^{(2)}\right)e^{ix\cdot\xi}a_{1}\overline{a}_{2}\,dx=0
\end{equation}
for all $\xi,\mu^{(1)},\mu^{(2)}\in\mathbb{R}^{n}$ with
\begin{equation}\label{perp}
\mu^{(1)}\cdot\mu^{(2)}=\xi\cdot\mu^{(1)}=\xi\cdot\mu^{(2)}=0,\quad |\mu^{(1)}|=|\mu^{(2)}|=1, \quad \mu^{(2)}_{n}=0, \quad \mu^{(1)}_{n}\neq 0.
\end{equation}
Replacing $\mu^{(1)}$ by $-\mu^{(1)}$ to get
\begin{equation}\label{mu2}
\left(\mu^{(1)}+i\mu^{(2)}\right)\cdot\displaystyle\int_{(\Sigma\cap B)\cup(\Sigma\cap B)^{\ast}_{0}}\left(\tilde{A}^{(1)}-\tilde{A}^{(2)}\right)e^{ix\cdot\xi}a_{1}\overline{a}_{2}\,dx=0.
\end{equation}
Hence, \eqref{mu1} and \eqref{mu2} imply that
\begin{equation}\label{mu3}
\mu\cdot\displaystyle\int_{(\Sigma\cap B)\cup(\Sigma\cap B)^{\ast}_{0}}\left(\tilde{A}^{(1)}-\tilde{A}^{(2)}\right)e^{ix\cdot\xi}a_{1}\overline{a}_{2}\,dx=0.
\end{equation}
for all $\mu\in\textrm{span}\{\mu^{(1)},\mu^{(2)}\}$ and all $\xi\in\mathbb{R}^{n}$ for which \eqref{perp} holds.\\

Next proposition indicates that $d\tilde{A}^{(1)}=d\tilde{A}^{(2)}$ in $(\Sigma\cap B)\cup(\Sigma\cap B)^{\ast}_{0}$.
\begin{prop}\label{curlvanish2}
\begin{equation} \label{curl2}
\partial_{j}(\tilde{A}^{(1)}_{k}-\tilde{A}^{(2)}_{k})-\partial_{k}(\tilde{A}^{(1)}_{j}-\tilde{A}^{(2)}_{j})=0 \textrm{ in } (\Sigma\cap B)\cup(\Sigma\cap B)^{\ast}_{0}, \quad 1\leq j,k\leq n.
\end{equation}
\end{prop}

\begin{proof}
If $n=3$, for any vector $\xi\in\mathbb{R}^{3}$ with $\xi^{2}_{1}+\xi^{2}_{2}>0$, it is easy to see the vectors
$$
\begin{array}{rl}\vspace{1ex}
\mu^{(1)}= & \displaystyle\frac{\tilde{\mu}^{(1)}}{|\tilde{\mu}^{(1)}|}\quad\quad
\tilde{\mu}^{(1)}=(-\xi_{1}\xi_{3},-\xi_{2}\xi_{3},\xi^{2}_{1}+\xi^{2}_{2}), \\ \vspace{1ex}
\mu^{(2)}= & \left(\displaystyle\frac{-\xi_{2}}{\sqrt{\xi^{2}_{1}+\xi^{2}_{2}}},
\displaystyle\frac{\xi_{1}}{\sqrt{\xi^{2}_{1}+\xi^{2}_{2}}},0\right),\\
\end{array}
$$
satisfy \eqref{perp}. Thus, after choosing $a_{1}=a_{2}=1$, \eqref{mu3} gives
\begin{equation}\label{muperp}
\mu\cdot f(\xi)=0 \quad \textrm{ where }f(\xi):=\widehat{\tilde{A}^{(1)}\chi}_{(\Sigma\cap B)\cup(\Sigma\cap B)^{\ast}_{0}}(\xi)-\widehat{\tilde{A}^{(2)}\chi}_{(\Sigma\cap B)\cup(\Sigma\cap B)^{\ast}_{0}}(\xi)
\end{equation}
for all $\mu\in\textrm{span}\{\mu^{(1)},\mu^{(2)}\}$. Here $\chi_{(\Sigma\cap B)\cup(\Sigma\cap B)^{\ast}_{0}}$ stands for the characteristic function of the set $(\Sigma\cap B)\cup(\Sigma\cap B)^{\ast}_{0}$. We decompose $f(\xi)\in\mathbb{R}^{3}$ as
$$f(\xi)=\alpha(\xi)\xi+f_{\perp}(\xi)$$
where Re$\,\alpha(\xi)$, Im$\,\alpha(\xi)$ are real numbers, and Re$\,f_{\perp}(\xi)$, Im$\,f_{\perp}(\xi)$ are orthogonal to $\xi$. As $n=3$, we conclude that Re$\,f_{\perp}(\xi)$, Im$\,f_{\perp}(\xi)\in\textrm{span}\{\mu^{(1)},\mu^{(2)}\}$. It follows from \eqref{muperp} that $f_{\perp}(\xi)=0$ for all $\xi\in\mathbb{R}^{3}$ with $\xi^{2}_{1}+\xi^{2}_{2}>0$. Hence $f(\xi)=\alpha(\xi)\xi$. Choose $\mu=-\xi_{k}e_{j}+\xi_{j}e_{k},1\leq j,k\leq 3, j\neq k$, where $e_{j}$ is the standard orthonormal basis of $\mathbb{R}^{3}$. This choice of $\mu$ satisfies $\mu\cdot f(\xi)=0.$ Therefore,
$$\xi_{j}\cdot(\widehat{\tilde{A}^{(1)}_{k}\chi(\xi)}-\widehat{\tilde{A}^{(2)}_{k}\chi(\xi)})-
\xi_{k}\cdot(\widehat{\tilde{A}^{(1)}_{j}\chi(\xi)}-\widehat{\tilde{A}^{(2)}_{j}\chi(\xi)})=0$$
for all $\xi\in\mathbb{R}^{3}$ with $\xi^{2}_{1}+\xi^{2}_{2}>0$, and hence everywhere by analyticity of the Fourier transform.

If $n\geq 4$, for any vector $\xi=(\xi_{1},\cdots,\xi_{n})\in\mathbb{R}^{n},\xi_{l}\neq 0,l=1,\cdots,n$, define vectors
$$\mu^{(1)}=(-\xi_{j}\xi_{n})e_{j}+(-\xi_{k}\xi_{n})e_{k}+(\xi^{2}_{j}+\xi^{2}_{k})e_{n}, \quad \mu^{(2)}=-\xi_{k}e_{j}+\xi_{j}e_{k}$$
where $1\leq j,k<n, j\neq k.$ It is easy to check that $\mu^{(1)}\cdot\mu^{(2)}=\mu^{(1)}\cdot\xi=\mu^{(2)}\cdot\xi=0,\mu^{(2)}_{n}=0$ and $\mu^{(1)}_{n}\neq 0$. Thus, after choosing $a_{1}=a_{2}=1$ and $\mu=\mu^{(2)}$, \eqref{mu3} implies
\begin{equation}\label{vanish4}
\xi_{j}\cdot(\widehat{\tilde{A}^{(1)}_{k}\chi(\xi)}-\widehat{\tilde{A}^{(2)}_{k}\chi(\xi)})-
\xi_{k}\cdot(\widehat{\tilde{A}^{(1)}_{j}\chi(\xi)}-\widehat{\tilde{A}^{(2)}_{j}\chi(\xi)})=0 \quad 1\leq j,k<n,j\neq k
\end{equation}
for all $\xi\in\mathbb{R}^{n}$ with $\xi_{l}\neq0,l=1,2,\cdots,n$.

Let $\xi=(\xi_{1},\cdots,\xi_{n})\in\mathbb{R}^{n}$ with $\xi_{l}\neq 0,l=1,2,\cdots,n,$ and let $1\leq j<n.$ Choose indices $k$ and $l$ so that the set $\{j,k,l,n\}$ consists of four distinct numbers. Define
$$\mu^{(1)}=-\xi_{n}e_{j}+\xi_{j}e_{n},\quad\quad \mu^{(2)}=-\xi_{k}e_{l}+\xi_{l}e_{k}.$$
Again one can check that $\mu^{(1)}\cdot\mu^{(2)}=\mu^{(1)}\cdot\xi=\mu^{(2)}\cdot\xi=0,\mu^{(2)}_{n}=0$ and $\mu^{(1)}_{n}\neq 0$. After choosing $a_{1}=a_{2}=1$ and $\mu=\mu^{(1)}$, \eqref{mu3} implies
\begin{equation}\label{vanish5}
\xi_{j}\cdot(\widehat{\tilde{A}^{(1)}_{n}\chi(\xi)}-\widehat{\tilde{A}^{(2)}_{n}\chi(\xi)})-
\xi_{n}\cdot(\widehat{\tilde{A}^{(1)}_{j}\chi(\xi)}-\widehat{\tilde{A}^{(2)}_{j}\chi(\xi)})=0 \quad 1\leq j<n
\end{equation}
for all $\xi\in\mathbb{R}^{n}$ with $\xi_{l}\neq0,l=1,2,\cdots,n$. The result in the case $n\geq 4$ then follows from \eqref{vanish4} and \eqref{vanish5}.
\end{proof}\vspace{1ex}

Arguing as in the proof of Theorem 1.1, we can find compactly supported function $\Phi\in C^{1,1}(\overline{\Sigma\cup\Sigma^{\ast}_{0}})$ such that
$$\tilde{A}^{(1)}-\tilde{A}^{(2)}=\nabla\Phi \quad\quad \textrm{ in } (\Sigma\cap B)\cup(\Sigma\cap B)^{\ast}_{0}$$
and $\Phi=0$ on $\partial((\Sigma\cap B)\cup(\Sigma\cap B)^{\ast}_{0})$. In \eqref{mu3}, pick $a_{2}=1$, $a_{1}$ satisfying
$$((i\mu^{(1)}+\mu^{(2)})\cdot\nabla)a_{1}=1 \quad\quad \textrm{ in } (\Sigma\cap B)\cup(\Sigma\cap B)^{\ast}_{0}.$$
and $\mu=i\mu^{(1)}+\mu^{(2)}$. Integrating by parts we obtain
$$
\begin{array}{rl}\vspace{1ex}
0= & (i\mu^{(1)}+\mu^{(2)})\cdot\displaystyle\int_{(\Sigma\cap B)\cup(\Sigma\cap B)^{\ast}_{0}}(\nabla\Phi)e^{ix\cdot\xi}a_{1}\,dx\\ \vspace{1ex}
=& \displaystyle\int_{(\Sigma\cap B)\cup(\Sigma\cap B)^{\ast}_{0}}\Phi(x)e^{ix\cdot\xi}[(i\mu^{(1)}+\mu^{(2)})\cdot\nabla a_{1}]\,dx\\ \vspace{1ex}
=& \displaystyle\int_{(\Sigma\cap B)\cup(\Sigma\cap B)^{\ast}_{0}}\Phi(x)e^{ix\cdot\xi}\,dx.
\end{array}
$$
This implies that $\Phi=0$ in $(\Sigma\cap B)\cup(\Sigma\cap B)^{\ast}_{0}$. Hence $\tilde{A}^{(1)}=\tilde{A}^{(2)}$, and therefore, $A^{(1)}=A^{(2)}$ in $\Sigma\cap B$.

As for electric potentials $q^{(1)}$ and $q^{(2)}$, continuing to argue as in the proof of Theorem 1.1, we arrive at
\begin{equation}\label{vanish6}
\displaystyle\int_{(\Sigma\cap B)\cup(\Sigma\cap B)^{\ast}_{0}}(q^{(1)}-q^{(2)})e^{ix\cdot\xi}\,dx=0
\end{equation}
for all $\mu^{(1)},\mu^{(2)},\xi\in\mathbb{R}^{n}$ satisfying \eqref{perp}. For any vector $\xi\in\mathbb{R}^{n}$ with $\xi^{2}_{n-2}+\xi^{2}_{n-1}>0$, the vectors
$$
\begin{array}{rl}\vspace{1ex}
\mu^{(1)}= & \displaystyle\frac{\tilde{\mu}^{(1)}}{|\tilde{\mu}^{(1)}|}\quad\quad
\tilde{\mu}^{(1)}=\left(0_{\mathbb{R}^{n-3}},\displaystyle\frac{-\xi_{n}\xi_{n-2}}{\sqrt{\xi^{2}_{n-2}+\xi^{2}_{n-1}}},
\displaystyle\frac{-\xi_{n}\xi_{n-1}}{\sqrt{\xi^{2}_{n-2}+\xi^{2}_{n-1}}},\displaystyle\sqrt{\xi^{2}_{n-2}+\xi^{2}_{n-1}}
\right),\\ \vspace{1ex}
\mu^{(2)}= & \left(0_{\mathbb{R}^{n-3}},\displaystyle\frac{-\xi_{n-1}}{\sqrt{\xi^{2}_{n-2}+\xi^{2}_{n-1}}},
\displaystyle\frac{\xi_{n-2}}{\sqrt{\xi^{2}_{n-2}+\xi^{2}_{n-1}}},0\right), \\
\end{array}
$$
satisfy \eqref{mu3}. Thus, \eqref{vanish6} holds for all $\xi\in\mathbb{R}^{n}$ with $\xi^{2}_{n-2}+\xi^{2}_{n-1}>0$. We conclude that \eqref{vanish6} also holds for all $\xi\in\mathbb{R}^{n}$ by the analyticity of the Fourier transform. This completes the proof of Theorem 1.2.\vspace{3ex}

\section{Proof of Theorem 1.3}
In this section we prove Theorem 1.3. Let $\Omega_{1}\subset\subset\Omega$ be a bounded sub-domain with $C^{\infty}$ boundary and be such that $\Omega\backslash\bar{\Omega}_{1}$ is connected and $supp(A^{(1)}-A^{(2)})$ and $supp(q^{(1)}-q^{(2)})$ are contained in $\Omega_{1}$.

Let $u_{1}\in H^{4}(\Omega)$ be the solution to the Dirichlet problem
$$\left\{
\begin{array}{rcll}
\mathcal{L}_{A^{(1)},q^{(1)}}u_{1} &=& 0 & \quad\textrm{ in } \Omega \\
u_{1}&=&f_{1} & \quad\textrm{ on } \partial\Omega \\
\Delta u_{1}&=&f_{2} & \quad\textrm{ on } \partial\Omega \\
\end{array}
\right.$$
with $(f_{1},f_{2})\in H^{\frac{7}{2}}(\partial\Omega)\times H^{\frac{3}{2}}(\partial\Omega)$ and $supp(f_{1})\subset\gamma_{1}, supp(f_{2})\subset\gamma_{1}$. Let $u_{2}\in H^{4}(\Omega)$ be the solution to the Dirichlet problem
$$\left\{
\begin{array}{rcll}
\mathcal{L}_{A^{(2)},q^{(2)}}u_{2} &=& 0 & \quad\textrm{ in } \Omega \\
u_{2}&=&f_{1} & \quad\textrm{ on } \partial\Omega \\
\Delta u_{2}&=&f_{2} & \quad\textrm{ on } \partial\Omega. \\
\end{array}
\right.$$
Setting $w=u_{2}-u_{1}$, then
$$\mathcal{L}_{A^{(2)},q^{(2)}}w=(A^{(1)}-A^{(2)})\cdot Du_{1}+(q^{(1)}-q^{(2)})u_{1} \textrm{ in } \Omega.$$
As two Dirichlet-to-Neumann maps agree on $\gamma_{2}$, we have $\partial_{\nu}w=0$ on $\gamma_{2}$. Therefore, $w$ is a solution of
$$\mathcal{L}_{A^{(2)},q^{(2)}}w=0 \quad \textrm{ in } \Omega\backslash\bar{\Omega}_{1}$$
with $w=\partial_{\nu}w=0$ on $\gamma_{2}$. By unique continuation, we obtain that $w=0$ in $\Omega\backslash\bar{\Omega}_{1}$. Thus, $w=\Delta w=\partial_{\nu}w=\partial_{\nu}\Delta w=0$ on $\partial\Omega_{1}$.

Let $v\in H^{4}(\Omega_{1})$ be a solution of
\begin{equation}\label{equation}
\mathcal{L}^{\ast}_{A^{(2)},q^{(2)}}v=0 \quad\textrm{ in } \Omega_{1}
\end{equation}
Using Green's formula \eqref{green} over $\Omega_{1}$, we have
\begin{equation}\label{difference2}
\displaystyle\int_{\Omega_{1}}((A^{(1)}-A^{(2)})\cdot Du_{1})\bar{v}\,dx+\int_{\Omega_{1}}(q^{(1)}-q^{(2)})u_{1}\bar{v}\,dx=0
\end{equation}
for all $v\in H^{4}(\Omega_{1})$ satisfying \eqref{equation} and for all $u_{1}\in \mathcal{W}(\Omega)$, where
$$\mathcal{W}(\Omega):=\{u\in H^{4}(\Omega):\mathcal{L}_{A^{(1)},q^{(1)}}u=0 \textrm{ in }\Omega, supp(u|_{\partial\Omega})\subset\gamma_{1},supp(\Delta u|_{\partial\Omega})\subset\gamma_{1}\}.$$
Let
$$\widetilde{\mathcal{W}}(\Omega_{1}):=\{u\in H^{4}(\Omega_{1}):\mathcal{L}_{A^{(1)},q^{(1)}}u=0 \textrm{ in }\Omega_{1}\}.$$
Again we need a density result to pass from $\mathcal{W}(\Omega)$ to $\widetilde{\mathcal{W}}(\Omega_{1})$.

\begin{prop}
$\mathcal{W}(\Omega)$ is a dense subspace in $\widetilde{\mathcal{W}}(\Omega_{1})$ in $L^{2}(\Omega_{1})$-topology.
\end{prop}
\begin{proof}
It suffices to establish the following fact: for any $g\in L^{2}(\Omega_{1})$ such that
$$\displaystyle\int_{\Omega_{1}} u\overline{g}\,dx=0 \quad\quad \forall u\in \mathcal{W}(\Omega),$$
we have
$$\displaystyle\int_{\Omega_{1}} v\overline{g}\,dx=0 \quad\quad \forall v\in \widetilde{\mathcal{W}}(\Omega).$$
To this end, extend $g$ by zero to $\Omega\backslash\Omega_{1}$. Let $U\in H^{4}(\Omega)$ be the solution of the Dirichlet problem
$$
\begin{array}{rl} \vspace{1ex}
\mathcal{L}^{\ast}_{A^{(1)},q^{(1)}}U=g & \quad \textrm{ in } \Omega\\ \vspace{1ex}
U=\Delta U=0 & \quad \textrm{ on } \partial\Omega.\\
\end{array}
$$
For any $u\in \mathcal{W}(\Omega)$, Green's formula on bounded domain $\Omega$ gives
$$
\begin{array}{rl}\vspace{1ex}
0=&\displaystyle\int_{\Omega}u\bar{g}\,dx=\displaystyle\int_{\Omega}u\overline{(\mathcal{L}^{\ast}_{A^{(1)},q^{(1)}}U)}\,dx \\ \vspace{1ex}
=&-\displaystyle\int_{\partial\Omega}(-\Delta u)\overline{\partial_{\nu}U}\,dS-\displaystyle\int_{\partial\Omega}u\overline{(\partial_{\nu}(-\Delta U))}\,dS\\
\end{array}
$$
where we have used $U=\Delta U=0$ on $\partial\Omega$. Since $u|_{\gamma_{1}}$ and $\Delta u|_{\gamma_{1}}$ can be arbitrary smooth functions supported in $\gamma_{1}$, we conclude that $\partial_{\nu}U|_{\gamma_{1}}=\partial_{\nu}\Delta U|_{\gamma_{1}}=0$. Hence $U$ satisfies $\mathcal{L}^{\ast}_{A^{(1)},q^{(1)}}U=0$ in $\Omega\backslash\Omega_{1}$, and $U=\Delta U=\partial_{\nu}U=\partial_{\nu}(\Delta U)=0$ on $\gamma_{1}$. By unique continuation, $U=0$ in $\Omega\backslash\Omega_{1}$, and therefore, $U=\Delta U=\partial_{\nu} U=\partial_{\nu}(\Delta U)=0$ on $\partial\Omega_{1}$.

For any $v\in \widetilde{\mathcal{W}}(\Omega_{1})$, using Green's formula over $\Omega_{1}$ we get
$$
\begin{array}{rl}\vspace{1ex}
\displaystyle\int_{\Omega_{1}}v\overline{g}\,dx= & \displaystyle\int_{\Omega_{1}}v\overline{(\mathcal{L}^{\ast}_{A^{(1)},q^{(1)}}U)}\,dx \\ \vspace{1ex}
=& \displaystyle\int_{\Omega_{1}}\left(\mathcal{L}_{A^{(1)},q^{(1)}}v\right)\overline{U}\,dx+i\int_{\partial\Omega_{1}}\nu(x)\cdot A\overline{U}v\,dS\\ \vspace{1ex}
& +\displaystyle\int_{\partial\Omega_{1}}\partial_{\nu}(-\Delta v)\overline{U}\,dS-\displaystyle\int_{\partial\Omega_{1}}(-\Delta v)\overline{\partial_{\nu}U}\,dS\\ \vspace{1ex}
& +\displaystyle\int_{\partial\Omega_{1}}\partial_{\nu}v\overline{(-\Delta U)}\,dS-\displaystyle\int_{\partial\Omega_{1}}v\overline{(\partial_{\nu}(-\Delta U))}\,dS\\ \vspace{1ex}
=& 0.\\
\end{array}
$$
\end{proof}
We conclude from this proposition that \eqref{difference2} holds for all $u\in\widetilde{\mathcal{W}}(\Omega_{1})$ and $v\in H^{4}(\Omega_{1})$ satisfying \eqref{equation}.

Let $B\subset\mathbb{R}^{n}$ be an open ball such that $\Omega_{1}\subset B$. The fact that $A^{(1)}=A^{(2)}$ and $q^{(1)}=q^{(2)}$ on $\partial\Omega_{1}$ allows to extend $A^{(j)}$ and $q^{(j)}$ to $B$ in such a way that the extensions, still denoted by $A^{(j)}$ and $q^{(j)}$, coincide on $B\backslash\Omega_{1}$, have compact supports, and satisfy $A^{(j)}\in W^{1,\infty}(B)$, $q^{(j)}\in L^{\infty}(B)$. It follows from \eqref{difference2} that
$$\displaystyle\int_{B}((A^{(1)}-A^{(2)})\cdot Du_{1})\bar{v}\,dx+\int_{B}(q^{(1)}-q^{(2)})u_{1}\bar{v}\,dx=0$$
for all $u_{1},v\in H^{4}(B)$ which are solutions of
$$\mathcal{L}_{A^{(1)},q^{(1)}}u_{1}=0 \textrm{ in } B \quad\quad\quad \mathcal{L}^{\ast}_{A^{(2)},q^{(2)}}v=0 \textrm{ in } B.$$
Now we are in the same situation as in \cite{KLU2} for the bi-harmonic operator, and as in \cite{KLU3} with full boundary measurements. We can construct complex geometric optics solutions as in Proposition \ref{existence}, and proceed as in \cite{KLU2}, \cite{KLU3} and the proof of Theorem 1.1 to show that $A^{(1)}=A^{(2)}$ and $q^{(1)}=q^{(2)}$ in $\Omega$.

\section{Proof of Theorem 1.4}
In this section we prove Theorem 1.4. First, as in the proof of Theorem 1.1 and Theorem 1.2, after applying Green's formula over $\Omega$, we obtain the integral identity
\begin{equation}\label{difference3}
\displaystyle\int_{\Omega}((A^{(1)}-A^{(2)})\cdot Du_{1})\bar{v}\,dx+\int_{\Omega}(q^{(1)}-q^{(2)})u_{1}\bar{v}\,dx=0
\end{equation}
for all $u_{1},v\in H^{4}(\Omega)$ such that
$$\mathcal{L}_{A^{(1)},q^{(1)}}u_{1}=0 \textrm{ in }\Omega,\quad\quad u_{1}|_{x_{n}=0}=(\Delta u_{1})|_{x_{n}=0}=0;$$
$$\mathcal{L}^{\ast}_{A^{(2)},q^{(2)}}v=0 \textrm{ in }\Omega,\quad\quad v|_{x_{n}=0}=(\Delta v)|_{x_{n}=0}=0.$$
Applying the reflection argument as in the proof of Theorem 1.2, we can construct CGO solutions $u_{1}$ and $v$, as in \eqref{CGO1} and \eqref{CGO3}, to the above equations and with the corresponding boundary conditions. Substituting these solutions $u_{1}$ and $v$ into \eqref{difference3} and proceeding as in the proof of Theorem 1.2 we get
\begin{equation}\label{mu4}
\left(\mu^{(1)}-i\mu^{(2)}\right)\cdot\displaystyle\int_{\Omega\cup\Omega^{\ast}_{0}}\left(\tilde{A}^{(1)}-\tilde{A}^{(2)}\right)e^{ix\cdot\xi}a_{1}\bar{a}_{2}\,dx=0
\end{equation}
for all $\xi,\mu^{(1)},\mu^{(2)}\in\mathbb{R}^{n}$ such that
$$\mu^{(1)}\cdot\mu^{(2)}=\xi\cdot\mu^{(1)}=\xi\cdot\mu^{(2)}=0,\quad |\mu^{(1)}|=|\mu^{(2)}|=1, \quad \mu^{(2)}_{n}=0, \quad \mu^{(1)}_{n}\neq 0,$$
where we have introduced the notation $\Omega^{\ast}_{0}:=\{(x',x_{n})\in\mathbb{R}^{n}:(x',-x_{n})\in\Omega\}$.
Applying the boundary reconstruction result \cite[Proposition 4.1]{KLU2} we conclude that $A^{(1)}=A^{(2)}$ on $\bar{\gamma}$, hence $\tilde{A}^{(1)}=\tilde{A}^{(2)}$ on $\partial(\Omega\cup\Omega^{\ast}_{0})$. This allows us to extend $\tilde{A}^{(j)},\; j=1,2,$ to compactly supported vector fields on a large ball $B$ with $\Omega\cup\Omega^{\ast}_{0}\subset\subset B$ and $\tilde{A}^{(1)}=\tilde{A}^{(2)}$ in $B\backslash\Omega\cup\Omega^{\ast}_{0}$. Then \eqref{mu4} leads to
$$\left(\mu^{(1)}-i\mu^{(2)}\right)\cdot\displaystyle\int_{B}\left(\tilde{A}^{(1)}-\tilde{A}^{(2)}\right)e^{ix\cdot\xi}a_{1}\bar{a}_{2}\,dx=0.$$
From Proposition \ref{curlvanish2} we have $d\tilde{A}^{(1)}=d\tilde{A}^{(2)}$ in $B$. Therefore, there exists $\Phi\in C^{1,1}(\overline{B})$ so that
$$\tilde{A}^{(1)}-\tilde{A}^{(2)}=\nabla\Phi \quad\quad\textrm{ in } B.$$
As before we can show that $\Phi=0$ on $\partial(\Omega\cup\Omega^{\ast}_{0})$; in particular, $\Phi=0$ on $\bar{\gamma}$. Now we are facing the same situation as in the proof of Theorem 1.2. Arguing as there we conclude that $A^{(1)}=A^{(2)}$ and $q^{(1)}=q^{(2)}$. This completes the proof of Theorem 1.4.

\begin{appendices}
\section{Solvability of the forward problem in an infinite slab}

In this appendix we provide the proof of the existence of the forward boundary value problem \eqref{Dirichlet1} for the perturbed bi-harmonic operator in an infinite slab. Recall that the perturbed bi-harmonic operator is of the form
$$\mathcal{L}_{A,q}(x,D):=\Delta^{2}+A(x)\cdot D+q(x).$$
The infinite slab is written as $(n\geq 3)$
$$\Sigma=\{x=(x',x_{n})\in\mathbb{R}^{n}:x'=(x_{1},\dots,x_{n-1})\in\mathbb{R}^{n-1}, 0<x_{n}<L\},\quad L>0.$$
whose boundary hyperplanes are
$$\Gamma_{1}=\{x\in\mathbb{R}^{n}:x_{n}=L\} \quad\quad \Gamma_{2}=\{x\in\mathbb{R}^{n}:x_{n}=0\}.$$
We will rewrite the perturbed bi-harmonic equation as a system of equations. For this purpose, let $u=(u_{1},u_{2})$ with $u_{2}=\Delta u_{1}$, define
$$\mathcal{S}u:=
\Delta \left(
           \begin{array}{c}
             u_{1} \\
             u_{2} \\
           \end{array}
         \right)+
         \left(
           \begin{array}{cc}
             0 & -1 \\
             A\cdot D+q & 0 \\
           \end{array}
         \right)
         \left(
           \begin{array}{c}
             u_{1} \\
             u_{2} \\
           \end{array}
         \right),
$$
then $\mathcal{L}_{A,q}u_{1}=0$ is equivalent to $\mathcal{S}u=0$. We will show the existence of a unique solution to this system with boundary value $(u_{1},u_{2})|_{\partial\Sigma}=(f_{1},f_{2})$.

Poincar\'{e}'s inequality in an infinite slab indicates that the quadratic form
$$u\mapsto \displaystyle\int_{\Sigma}|\nabla u|^{2}\,dx=\displaystyle\int_{\Sigma}(|\nabla u_{1}|^{2}+|\nabla u_{2}|^{2})\,dx$$
is non-negative and densely defined closed on $H^{1}_{0}(\Sigma)\times H^{1}_{0}(\Sigma)$. Associated with this quadratic form, the Laplace operator $-\Delta$ equipped with the domain
$$\mathcal{D}(-\Delta):=\{u\in H^{1}_{0}(\Sigma)\times H^{1}_{0}(\Sigma):\Delta u=(\Delta u_{1},\Delta u_{2})\in L^{2}(\Sigma)\times L^{2}(\Sigma)\}$$
is a non-negative self-adjoint operator on $L^{2}(\Sigma)\times L^{2}(\Sigma)$. Its spectrum is obtained in the following proposition.

\begin{prop}
$\mathcal{D}(-\Delta)=H^{1}_{0}(\Sigma)\cap H^{2}(\Sigma) \times H^{1}_{0}(\Sigma)\cap H^{2}(\Sigma)$. Moreover, the spectrum of $-\Delta$ is purely absolutely continuous and is equal to $[\pi^{2}/L^{2},+\infty)$.
\end{prop}
\begin{proof}
Let $F=(F_{1},F_{2})\in L^{2}(\Sigma)\times L^{2}(\Sigma)$, we will consider
$$-\Delta u=F, \quad\quad u\in\mathcal{D}(-\Delta).$$
Taking the Fourier series with respect to the variable $x_{n}\in [0,L]$ we have
\begin{equation}\label{Fourier}
\begin{array}{lr}
u(x',x_{n})=\displaystyle\sum^{\infty}_{l=1}u_{l}(x')\sin\frac{l\pi x_{n}}{L}, &\quad u_{l}(x')=\displaystyle\frac{2}{L}\displaystyle\int^{L}_{0}u(x)\sin\frac{l\pi x_{n}}{L}\,dx_{n};\\
F(x',x_{n})=\displaystyle\sum^{\infty}_{l=1}F_{l}(x')\sin\frac{l\pi x_{n}}{L}, &\quad F_{l}(x')=\displaystyle\frac{2}{L}\displaystyle\int^{L}_{0}F(x)\sin\frac{l\pi x_{n}}{L}\,dx_{n}.\\
\end{array}
\end{equation}
As usual Parseval's identities hold
$$
\begin{array}{rl}
\|u\|_{L^{2}(\Sigma)\times L^{2}(\Sigma)} & =\displaystyle\frac{L}{2}\displaystyle\sum^{\infty}_{l=1}\|u_{l}\|^{2}_{L^{2}(\mathbb{R}^{n-1})\times L^{2}(\mathbb{R}^{n-1})},\\
\|F\|_{L^{2}(\Sigma)\times L^{2}(\Sigma)} & =\displaystyle\frac{L}{2}\displaystyle\sum^{\infty}_{l=1}\|F_{l}\|^{2}_{L^{2}(\mathbb{R}^{n-1})\times L^{2}(\mathbb{R}^{n-1})}.\\
\end{array}
$$
Comparing the Fourier coefficients $u_{l}$ of $u$ and $F_{l}$ of $F$ we see that they are related by
\begin{equation}\label{Fourieridentity}
\left(-\Delta_{x'}+\displaystyle\frac{l^{2}\pi^{2}}{L^{2}}\right)u_{l}(x')=F_{l}(x'), \quad\quad x'\in\mathbb{R}^{n-1}, l=1,2,\dots.
\end{equation}
The operator $-\Delta_{x'}+\frac{l^{2}\pi^{2}}{L^{2}}$ $(l\geq 1)$, when equipped with the domain $H^{2}(\mathbb{R}^{n-1})$, is self-adjoint on $L^{2}(\mathbb{R}^{n-1})$ with purely absolutely continuous spectrum $[l^{2}\pi^{2}/L^{2},+\infty)$. Hence \eqref{Fourieridentity} has the unique solution
$$u_{l}(x')=\left(-\Delta_{x'}+\displaystyle\frac{l^{2}\pi^{2}}{L^{2}}\right)^{-1}F_{l}(x')\in H^{2}(\mathbb{R}^{n-1}),$$
and moreover, it satisfies the norm estimate
\begin{equation}\label{normestimate1}
\begin{array}{rl}\vspace{1ex}
\|u_{l}\|_{L^{2}(\mathbb{R}^{n-1})\times L^{2}(\mathbb{R}^{n-1})} & \leq\displaystyle\frac{L^{2}}{l^{2}\pi^{2}}\|F_{l}\|_{L^{2}(\mathbb{R}^{n-1})\times L^{2}(\mathbb{R}^{n-1})};\\\vspace{1ex}
\|u_{l}\|_{H^{2}(\mathbb{R}^{n-1})\times H^{2}(\mathbb{R}^{n-1})} & \leq C\|F_{l}\|_{L^{2}(\mathbb{R}^{n-1})\times L^{2}(\mathbb{R}^{n-1})}.\\
\end{array}
\end{equation}
Here and in the following we will name all the constants independent of $l$ as C. By interpolation we obtain
\begin{equation}\label{normestimate2}
\|u_{l}\|_{H^{1}(\mathbb{R}^{n-1})\times H^{1}(\mathbb{R}^{n-1})} \leq \displaystyle\frac{C}{l}\|F_{l}\|_{L^{2}(\mathbb{R}^{n-1})\times L^{2}(\mathbb{R}^{n-1})}.
\end{equation}
Parseval's identities and \eqref{normestimate1} then give
$$
\begin{array}{rl}\vspace{1ex}
\|u\|^{2}_{L^{2}(\Sigma)\times L^{2}(\Sigma)} = & \displaystyle\frac{L}{2}\displaystyle\sum^{\infty}_{l=1}\|u_{l}\|^{2}_{L^{2}(\mathbb{R}^{n-1})\times L^{2}(\mathbb{R}^{n-1})}\\\vspace{1ex}
\leq & C \displaystyle\sum^{\infty}_{l=1}\frac{1}{l^{4}}\|F_{l}\|^{2}_{L^{2}(\mathbb{R}^{n-1})\times L^{2}(\mathbb{R}^{n-1})}\leq C \|F\|^{2}_{L^{2}(\Sigma)\times L^{2}(\Sigma)}.\\
\end{array}
$$
To take care of the first order derivatives, we differentiate with respect to $x_{n}$ to get
$$
\begin{array}{rl}\vspace{1ex}
\|\partial_{x_{n}}u\|^{2}_{L^{2}(\Sigma)\times L^{2}(\Sigma)} = &
\|\displaystyle\sum^{\infty}_{l=1}\displaystyle\frac{l\pi}{L}u_{l}(x')\cos\frac{l\pi x_{n}}{L}\|^{2}_{L^{2}(\Sigma)\times L^{2}(\Sigma)}\\\vspace{1ex}
= & \displaystyle\frac{L}{2}\displaystyle\sum^{\infty}_{l=1}\frac{l^{2}\pi^{2}}{L^{2}}\|u_{l}\|^{2}_{L^{2}(\mathbb{R}^{n-1})\times L^{2}(\mathbb{R}^{n-1})}\leq C \|F\|^{2}_{L^{2}(\Sigma)\times L^{2}(\Sigma)}.\\
\end{array}
$$
Using \eqref{normestimate2} we obtain that for $j=1,2,\cdots,n-1$,
$$
\begin{array}{rl}\vspace{1ex}
\|\partial_{x_{j}}u\|^{2}_{L^{2}(\Sigma)\times L^{2}(\Sigma)} = & \displaystyle\frac{L}{2}\displaystyle\sum^{\infty}_{l=1}\|\partial_{x_{j}}u_{l}\|^{2}_{L^{2}(\mathbb{R}^{n-1})\times L^{2}(\mathbb{R}^{n-1})}\\\vspace{1ex}
\leq & C \displaystyle\sum^{\infty}_{l=1}\frac{1}{l^{2}}\|F_{l}\|^{2}_{L^{2}(\mathbb{R}^{n-1})\times L^{2}(\mathbb{R}^{n-1})}\leq C \|F\|^{2}_{L^{2}(\Sigma)\times L^{2}(\Sigma)}.\\
\end{array}
$$
We proceed to estimate the second order derivatives. For $j,k=1,2,\cdots,n-1$, it follows from \eqref{normestimate1} that
$$
\begin{array}{rl}\vspace{1ex}
\|\partial^{2}_{x_{j}x_{k}}u\|^{2}_{L^{2}(\Sigma)\times L^{2}(\Sigma)} = & \displaystyle\frac{L}{2}\displaystyle\sum^{\infty}_{l=1}\|\partial_{x_{j}x_{k}}u_{l}\|^{2}_{L^{2}(\mathbb{R}^{n-1})\times L^{2}(\mathbb{R}^{n-1})}\leq C \|F\|^{2}_{L^{2}(\Sigma)\times L^{2}(\Sigma)};\\ \vspace{1ex}
\|\partial^{2}_{x_{j}x_{n}}u\|^{2}_{L^{2}(\Sigma)\times L^{2}(\Sigma)} = & \displaystyle\frac{L}{2}\displaystyle\sum^{\infty}_{l=1}\frac{l^{2}\pi^{2}}{L^{2}}\|\partial_{x_{j}}u_{l}\|^{2}_{L^{2}(\mathbb{R}^{n-1})\times L^{2}(\mathbb{R}^{n-1})}\leq C \|F\|^{2}_{L^{2}(\Sigma)\times L^{2}(\Sigma)};\\ \vspace{1ex}
\|\partial^{2}_{x_{n}}u\|^{2}_{L^{2}(\Sigma)\times L^{2}(\Sigma)} = & \displaystyle\frac{L}{2}\displaystyle\sum^{\infty}_{l=1}\frac{l^{4}\pi^{4}}{L^{4}}\|u_{l}\|^{2}_{L^{2}(\mathbb{R}^{n-1})\times L^{2}(\mathbb{R}^{n-1})}\leq C \|F\|^{2}_{L^{2}(\Sigma)\times L^{2}(\Sigma)}.\\
\end{array}
$$
These estimates show that $u\in H^{2}(\Sigma)\times H^{2}(\Sigma)$. The statement concerning the spectrum of $-\Delta$ follows from the fact that
$$-\Delta=\bigoplus^{\infty}_{l=1}\left(-\Delta_{x'}+\displaystyle\frac{l^{2}\pi^{2}}{L^{2}}\right).$$
This completes the proof of the proposition.
\end{proof}

\begin{prop}
Let $A\in W^{1,\infty}({\Sigma};\mathbb{C}^{n})\cap\mathcal{E}'(\bar{\Sigma};\mathbb{C}^{n})$, $q\in L^{\infty}(\Sigma;\mathbb{C})\cap\mathcal{E}'(\bar{\Sigma};\mathbb{C}^{n}).$ Then the operator $\mathcal{S}$, equipped with the domain $H^{1}_{0}(\Sigma)\cap H^{2}(\Sigma) \times H^{1}_{0}(\Sigma)\cap H^{2}(\Sigma)$, is closed and its essential spectrum is equal to $[\pi^{2}/L^{2},+\infty)$.
\end{prop}
\begin{proof}
This follows from the fact that
$$
\left(
    \begin{array}{cc}
        0 & -1 \\
        A\cdot D+q & 0 \\
    \end{array}
\right)\Delta^{-1}:L^{2}(\Sigma)\times L^{2}(\Sigma)\rightarrow L^{2}(\Sigma)\times L^{2}(\Sigma)
$$
is a compact operator and that the essential spectrum do not change under relatively compact perturbations.
\end{proof}

This proposition yields the following solvability result. Suppose $A\in W^{1,\infty}({\Sigma};\mathbb{C}^{n})\cap\mathcal{E}'(\bar{\Sigma};\mathbb{C}^{n})$ and $q\in L^{\infty}(\Sigma;\mathbb{C})\cap\mathcal{E}'(\bar{\Sigma};\mathbb{C}^{n})$, then for any $F=(F_{1},F_{2})\in L^{2}(\Sigma)\times L^{2}(\Sigma)$, the boundary value problem
\begin{equation}\label{Dirichlet3}
\left\{
\begin{array}{rll}\vspace{1ex}
\mathcal{S}u= & F &\quad\textrm{ in } \Sigma\times\Sigma \\ \vspace{1ex}
u|_{\partial\Sigma\times\partial\Sigma}= & 0. & \\
\end{array}
\right.
\end{equation}
admits a unique solution $u\in H^{2}(\Sigma)\times H^{2}(\Sigma)$.

Given any $f=(f_{1},f_{2})\in (H^{\frac{7}{2}}(\Gamma_{1})\cap\mathcal{E}'(\Gamma_{1}))\times (H^{\frac{3}{2}}(\Gamma_{1})\cap\mathcal{E}'(\Gamma_{1}))$, we can establish the existence and uniqueness of the solution to the boundary value problem \eqref{Dirichlet1} as follows. \eqref{Dirichlet1} is equivalent to the following boundary value problem for the system
\begin{equation}\label{Dirichlet4}
\left\{
\begin{array}{rll}\vspace{1ex}
\mathcal{S}u= & 0 &\quad\textrm{ in } \Sigma\times\Sigma \\ \vspace{1ex}
u|_{\Gamma_{1}\times\Gamma_{1}}= & f & \\ \vspace{1ex}
u|_{\Gamma_{2}\times\Gamma_{2}}= & 0. & \\
\end{array}
\right.
\end{equation}
Uniqueness of the solution to \eqref{Dirichlet4} follows from the unique solvability of \eqref{Dirichlet3} when $F=0$. To show that \eqref{Dirichlet4} has at least one solution, choose $G\in H^{4}(\Sigma)\cap\mathcal{E}'(\overline{\Sigma})\times H^{2}(\Sigma)\cap\mathcal{E}'(\overline{\Sigma})$ so that $G|_{\Gamma_{1}\times\Gamma_{1}}=f$ and $G|_{\Gamma_{2}\times\Gamma_{2}}=0$; choose $u_{0}$ to be the unique solution of \eqref{Dirichlet3} when $F=-\mathcal{S}G$, then $G+u_{0}$ is a solution for \eqref{Dirichlet4}. This completes the proof that \eqref{Dirichlet1} admits a unique solution in $H^{4}(\Sigma)$ for $f_{1}$ and $f_{2}$.

\section{Green's formula in a slab}

In the proof of Proposition \ref{runge} we used the Green's formula in a slab, in this part we establish this identity.
For $R>0$, define $\Sigma_{R}$ by
$$\Sigma_{R}:=\{x\in\Sigma:|x'|<R\}.$$
We may choose $R>0$ sufficiently large so that $supp(A^{(1)})\subset\overline{\Sigma_{R}}$. Introduce the notations
$$d_{j}(R):=\partial\Sigma_{R}\cap\Gamma_{j}, j=1,2;\quad\quad d_{3}(R)=\partial\Sigma_{R}\cap\Sigma,$$
then $A^{(1)}=0$ on $d_{3}(R)$. Let $u\in W(\Sigma)$ and let $U\in H^{4}(\Sigma)$ be the solution of the problem
$$
\begin{array}{rl} \vspace{1ex}
\mathcal{L}^{\ast}_{A^{(1)},q^{(1)}}U=g & \quad \textrm{ in } \Sigma\\ \vspace{1ex}
U=\Delta U=0 & \quad \textrm{ on } \Gamma_{1}\cup\Gamma_{2}.\\
\end{array}
$$
Apply Green's formula \eqref{green} over the region $\Sigma_{R}$ we obtain
\begin{equation}\label{greenslab}
\begin{array}{rl}\vspace{1ex}
& -\displaystyle\int_{\Sigma_{R}}ug\,dx\\ \vspace{1ex}
=& \displaystyle\int_{\Sigma_{R}}(\mathcal{L}_{A^{(1)},q^{(1)}}u)\overline{U}\,dx-\int_{\Sigma_{R}}u\overline{(\mathcal{L}^{\ast}_{A^{(1)},q^{(1)}}U)}\,dx\\ \vspace{1ex}
=& -\displaystyle\int_{d_{3}(R)}\partial_{\nu}(-\Delta u)\overline{U}\,dS+\int_{d_{1}(R)\cup d_{3}(R)}(-\Delta u)\overline{\partial_{\nu}U}\,dS\\ \vspace{1ex}
 & -\displaystyle\int_{d_{3}(R)}\partial_{\nu}u\overline{(-\Delta U)}\,dS+\int_{d_{1}(R)\cup d_{3}(R)}u\overline{\partial_{\nu}(-\Delta U)}\,dS.\\
\end{array}
\end{equation}
We will show that the right hand side converges to
\begin{equation}\label{righthandside}
\int_{\Gamma_{1}}\overline{\partial_{\nu}U}\Delta u\,dS+\int_{\Gamma_{1}}\overline{\partial_{\nu}\Delta U} u\,dS.
\end{equation}

To this end, notice that for $R>0$ sufficiently large,
$$\Delta^{2}u=\Delta^{2}U=0 \quad\quad \textrm{ in } \Sigma\backslash\Sigma_{R},$$
$$\Delta u=\Delta U=0 \quad\quad \textrm{ on } \partial(\Sigma\backslash\Sigma_{R}).$$
According to \cite{MW}, we have
\begin{equation}
\Delta u, \partial_{\nu}(\Delta u), \Delta U, \partial_{\nu}(\Delta U) \textrm{ are of order } \mathcal{O}(|x'|^{-n}) \textrm{ as } |x'|\rightarrow\infty.
\end{equation}
We can estimate the first term on the right hand side of \eqref{greenslab} as follows
$$
\begin{array}{rl}\vspace{1ex}
  & \left|-\displaystyle\int_{d_{3}(R)}\partial_{\nu}(-\Delta u)\overline{U}\,dS\right| = \left|\displaystyle\int_{|x'|=R,0<x_{n}<L}\partial_{\nu}(\Delta u)\overline{U}\,dS\right|\\ \vspace{1ex}
\leq & \left(\displaystyle\int_{|x'|=R,0<x_{n}<L}|\partial_{\nu}(\Delta u)|^{2}\,dS\right)^{\frac{1}{2}} \left(\displaystyle\int_{|x'|=R,0<x_{n}<L}|U|^{2}\,dS\right)^{\frac{1}{2}}\\ \vspace{1ex}
\leq & \mathcal{O}(R^{-\frac{n}{2}-1})\left(\displaystyle\int_{|x'|=R,0<x_{n}<L}|U|^{2}\,dS\right)^{\frac{1}{2}}\\ \vspace{1ex}
\leq & \mathcal{O}(R^{-\frac{n}{2}-1})\left(\displaystyle\int_{\partial\Sigma_{R}}|U|\,dS\right)^{\frac{1}{2}}\\ \vspace{1ex}
\leq & C \mathcal{O}(R^{-\frac{n}{2}-1})\|U\|_{H^{2}(\Sigma)} \rightarrow 0 \textrm{ as } R\rightarrow\infty.\\
\end{array}
$$
where in the last step the constant $C$ comes from the trace theorem. Similarly, all the other terms involving $d_{3}(R)$ on the right hand side of \eqref{greenslab} will vanish as $R\rightarrow\infty$. Therefore, after taking the limit $R\rightarrow\infty$ in \eqref{greenslab}, the right hand side will become \eqref{righthandside}, as we have claimed.\\

\end{appendices}

\begin{center}
ACKNOWLEDGEMENT
\end{center}
The author would like to thank Professor Gunther Uhlmann for his constant encouragement and support. The author would also like to thank Dr. Katya Krupchyk for her assistance and helpful discussions. This work is partially supported by the NSF grant DMS 126598.

\end{document}